\documentclass[12pt,a4paper]{article}%
\usepackage{amssymb}
\usepackage{amsmath}
\usepackage{amsfonts}
\usepackage[dvips]{graphicx}
\usepackage{mathrsfs}
\usepackage[pdftex]{hyperref}
\usepackage{pdfpages}%
\usepackage[all]{xy}
\providecommand{\U}[1]{\protect\rule{.1in}{.1in}}
\hypersetup{
	colorlinks=true,
	citecolor = blue
}
\newtheorem{theorem}{Theorem}

\newtheorem{corollary}[theorem]{Corollary}

\newtheorem{definition}[theorem]{Definition}
\newtheorem{example}[theorem]{Example}

\newtheorem{lemma}[theorem]{Lemma}

\newtheorem{proposition}[theorem]{Proposition}
\newtheorem{remark}[theorem]{Remark}

\newenvironment{proof}[1][Proof]{\noindent\textbf{#1.} }{\ \rule{0.5em}{0.5em}}
\def\hc{\mathbb{H}_{\mathbb{C}}}
\def\sc{\operatorname{Sc}}
\def\imag{\operatorname{i}}
\def\dive{\operatorname{div}}
\def\rot{\operatorname{curl}}
\def\grad{\operatorname{grad}}
\begin{document}
	
	\author{V\'{\i}ctor A. Vicente-Ben\'{\i}{}tez\\{\small Instituto de Matemáticas de la U.N.A.M. Campus Juriquilla}\\{\small Boulevard Juriquilla 3001, Juriquilla, Querétaro C.P. 076230 M\'{e}xico }\\{\small  va.vicentebenitez@im.unam.mx } }
	\title{Bergman spaces and reproducing kernel for the biquaternionic Vekua equation}
	\date{}
	\maketitle
	\begin{abstract}
		We analyze the main properties of the Bergman spaces of weak $L_p$- solutions for a biquaternionic Vekua equation of the form 
		\[
		\mathbf{D}w(x)-\mathbf{Q}_Aw(x)=0
		\]
		on bounded domains of $\mathbb{R}^3$, where the operator $\mathbf{Q}_A$ involves quaternionic conjugation and  multiplications, both left and right, by essentially bounded functions. Properties such as completeness, separability, and reflexivity are shown. It is demonstrated that the solutions belonging to the Bergman spaces are locally H\"older continuous and that the evaluation maps are bounded in the $L_p$-norm. Consequently, for the case $p=2$, we obtain a reproducing integral kernel and an explicit formula for the orthogonal projection onto the Bergman space. For $1<p<\infty$, the explicit form for the annihilator of the Bergman space in the dual $L_{p'}$ is presented, along with an orthogonal decomposition for $L_2$.
	\end{abstract}

    
    
	\section{Introduction}
	The aim of this paper is to develop and analyze the main properties of the Bergman space of $L_p$-weak solutions of the biquaternionic Vekua equation
	\begin{equation}\label{eq:vekuaintro}
		\mathbf{D}w(x)-w(x)a_1(x)+\overline{w(x)}a_2(x)+a_3(x)w(x)+a_4(x)\overline{w(x)}=0, \quad x\in G,
    \end{equation}	
   where $w$ and $a_j$, $j=1,2,3,4$, are biquaternionic-valued functions, $G\subset \mathbb{R}^3$ is a bounded domain with a Liapunov boundary, and $\mathbf{D}$ is the Moisil-Theodorescu operator (also called Cauchy-Riemann or Dirac operator). The notation $\overline{w}$ stands for the quaternionic conjugate of $w$. We assume that $a_j$, $j=1,2,3,4$, are essentially bounded on $G$.
   
   Biquaternionic equations of the form \eqref{eq:vekuaintro} arise in several elliptic systems of mathematical physics, such as the Dirac and Maxwell systems \cite{castillokrav,kravfather,kravmaxwell}, div-curl systems \cite{delgadodivcurl}, the study of Beltrami fields \cite{kravbeltrami,kravpablo}, and the factorization of Helmholtz and Schr\"odinger equations \cite{berstein,kravchenkoricatti,kravquaternionic}. This type of equation was treated first in the complex case by I. Vekua \cite{vekua0} and L. Bers \cite{bers} (who developed the theory of {\it pseudoanalytic functions}, called by Vekua {\it Generalized analytic functions}). The theory of pseudoanalytic functions and $(F,G)$-derivatives (in the sense of Bers) was extended to quaternionic-valued functions in \cite{malonek,malonek2}. In \cite{kravshapiro,shapirokrav}, the theory of fundamental solutions and integral representations was developed for the case when $a_1$ is a quaternionic constant and $a_2=a_3=a_4=0$. Eq. \eqref{eq:vekuaintro} is called a {\it main Vekua equation} when, for a fixed $j\in \{1,2,3,4\}$,  $a_j=\frac{\nabla f}{f}$ and $a_k=0$ for $j\neq k$, where $f$ is a $C^1$ function that does not vanish in the whole domain (the terminology extends the complex case given in \cite[Ch. 3]{pseudoanalyticvlad}). The cases for $j\in \{1,4\}$ are associated with the factorization of the equations $\operatorname{div}f^2\nabla \left(\frac{u}{f}\right)=0$ and $-\bigtriangleup u +\frac{\bigtriangleup f}{f}u=0$, which have been studied in works such as \cite{berstein,kravtrem,kravfather,delgadodivcurl}, focusing on classical and $W^{1,2}$-solutions. In \cite{delgadohardy}, a generalization of Hardy spaces for the main Vekua equation in the case $j=4$ was studied. The analysis of fundamental solutions for the case $j=3$ was presented in \cite{sprossing,sprossig2}.
   
   In this work, we analyze the main properties of the weak $L_p$- solutions (in the sense of Vekua \cite{gurlebeckquater,vekua0}) of \eqref{eq:vekuaintro}, such as completeness, separability, and reflexivity. In the first part, we investigate the regularity of the solutions to establish a result concerning the boundedness in the $L_p$-norm of the evaluation map $w\mapsto w(x)$, where $x\in G$. Consequently, for the case $p=2$, we prove the existence of integral kernels $K_x^k(t)$ satisfying $w_k(x)=\langle K_x^k,w\rangle_{L_2}$. Unlike the complex case, where the use of the similarity principle is crucial in proving the existence of the reproducing kernel \cite{camposbergman}, there is no result of this nature in the biquaternionic case. Recently, in \cite{delgadohodge}, the existence of reproducing integral kernels for an equation with $a_1=a_2=0$  involving $Cl_{0,n}$-valued functions was established. However, this result holds under the hypothesis that $\|a_3\|_{L_{\infty}}+\|a_4\|_{L_{\infty}}<\frac{1}{M}$, where $M$ is the operator-norm of the Theodorescu transform (defined in Subsection \ref{Subsmonogenicint}).  Generalizing the ideas presented in \cite{mineVekua2} for the bicomplex Vekua equation, we study the regularity of the solutions of \eqref{eq:vekuaintro} and establish their H\"older continuity on every ball contained within the domain $G$. Employing a Calderon-Zygmun type inequality, along with the integral operator $\mathbf{S}_G^A=\mathbf{I}+\mathbf{T}_G\mathbf{Q}_A$, where $\mathbf{T}_G$ is the Theodorescu transform and $\mathbf{Q}_A=wa_1+\overline{w}a_2+a_3w+a_4\overline{w}$ (this kind of operator was used in other works such as \cite{camposbergman,delgadohodge,mineVekua2}), and the Sobolev embedding theorems, we prove the boundedness of the evaluation map $w\mapsto w(x)$ in the $L_p$-norm. This result is proven without assuming more than the functions $a_j$, $j=1,2,3,4$, are essentially bounded on $G$. As a consequence, for $p=2$, we obtain an integral reproducing kernel for each component function $w_k(x)$ and a representation for the orthogonal projection of $L_2(G;\hc)$ onto the the Bergman space.
   
   The second part of this work is dedicated to establishing the annihilator of the Bergman space for $1<p<\infty$. We compute the transpose (in the sense of \cite{brezis}) of the operator $\mathbf{D}-\mathbf{Q}_A$, which is given by $\mathbf{D}-\mathbf{Q}_{A^{\star}}$, where $\mathbf{Q}_{A^{\star}}w=w\overline{a_1}+\overline{w}a_4+\overline{a_3}w+a_2\overline{w}$. The domain in $L_{p'}(G;\hc)$ of the transpose is the Sobolev space $W_0^{1,p'}(G;\hc)$. From this result, we obtain that the annihilator of the Bergman space is the closure in $L_{p'}$ of $(\mathbf{D}-\mathbf{Q}_{A^{\star}})W_0^{1,p'}(G;\hc)$. In particular, this yields an orthogonal decomposition for the space $L_2(G;\hc)$. This type of decomposition is useful for solving boundary value problems associated with equations that can be factorized in terms of operators of the type $\mathbf{D}-\mathbf{Q}_A$ (see \cite{gurlebeckquater,gurlebeckclifford} for examples in the case when $a_1$ is a constant and $a_2=a_3=a_4=0$). Knowledge of the annihilator is also useful for the problem of establishing the dual of the Bergman space \cite{cofman,hedenmal}.
   
   The paper is structured as follows. Section 2 presents the basic properties of the algebra of biquaternions, monogenic functions and integral operators. Section 3  establishes the main properties of the Bergman spaces of generalized solutions of Eq. \eqref{eq:vekuaintro}. Section 4 focuses on the regularity of the solutions and demonstrates the boundedness of the evaluation map with respect to the $L_p$-norm. Section 5 is devoted to the construction of the Bergman reproducing kernel and orthogonal projection. Section 6 derives the explicit expression of the annihilator of the Bergman space, provides an orthogonal decomposition of the space $L_2(G;\hc)$, and discusses some examples.
	\section{Background on quaternionic analysis}
	This section provides a summary of the main properties of the algebra of biquaternions and biquaternionic-valued holomorphic functions on $\mathbb{R}^3$, as discussed in \cite{gurlebeckquater,kravquaternionic}. We use the notation $\mathbb{N}_0=\mathbb{N}\cup \{0\}$. Given Banach spaces $X$ and $Y$, $\mathcal{B}(X,Y)$ denotes the Banach space of bounded linear operators and $\mathcal{K}(X,Y)$ the subspace of compact operators. When $X=Y$, we denote $\mathcal{B}(X)=\mathcal{B}(X,X)$ and $\mathcal{K}(X)=\mathcal{K}(X,Y)$. Additionally, $\mathcal{G}(X)$ denotes the group of invertible operators. 
	\subsection{Quaternions and quaternionic-valued functions}
	The algebra of the biquaternions (or complex quaternions) is denoted by $\hc$. The quaternionic units are denoted by $e_k$, $k=0,1,2, 3$, and satisfy the relations $e_0e_j=e_je_0=e_j$ and $e_je_k+e_ke_j=2\delta_{j,k} e_0$ for $j,k=1,2,3$.  Additionally, we have the relation with the complex imaginary unit: $\imag e_k=e_k\imag$, $k=0,1,2,3$ . We identify $e_0$ with $1$, and a biquaternion $p\in \hc$ is written as $p=p_0+\sum_{k=1}^{3}p_ke_k$, where $\{p_k\}_{k=0}^{3}\subset \mathbb{C}$. The complex number $p_0$ is called the {\it scalar part} of $p$ and is denoted by $\sc p$. The {\it vector part} of $p$ is $\vec{p}:=\sum_{k=1}^{3}p_ke_k$ (sometimes denoted by $\operatorname{Vec}p$). We identify $\vec{p}$ with a vector in $\mathbb{C}^3$. Hence $p=p_0+\vec{p}$, and the product of two quaternions $p$ and $q$ has the form
	\[
	pq=p_0q_0+p_0\vec{q}+q_0\vec{p}-\vec{p}\cdot \vec{q}+\vec{p}\times \vec{q},
	\]
	where $\vec{p}\cdot \vec{q}=\sum_{k=1}^{3}p_kq_k$ and $\vec{p}\times \vec{q}$ is the cross product of $\vec{p}$ with $\vec{q}$. The algebra $\hc$ is non-commutative and contains zero divisors \cite[Sec. 2.2]{kravquaternionic}. Given $p\in \hc$, its quaternionic conjugate is defined as $\overline{p}:=p_0-\vec{p}$. For $p\in \hc$, define $p^{\dagger}:= p_0^*-\vec{p}^*=p_0-\sum_{k=1}^3p_k^*e_k$, 
	where $^{*}$ denotes the standard complex conjugation. The operation $p\mapsto p^{\dagger}$ is an involution in $\hc$ and 
	\[
	\langle p,q\rangle _{\hc}:=\sc (p^{\dagger}q)=\sum_{k=0}^{3}p_k^*q_k \quad  \mbox{ and }\quad  |p|_{\hc}:=\sqrt{\langle p,p\rangle_{\hc} }
	\]
	are an inner product (conjugate-linear in the first slot) and a norm in $\hc$. For $p,q\in \hc$ we have the relation \cite[Lemma 2 of Sec. 2.2]{kravquaternionic}
	\begin{equation}\label{normproduct}
		|pq|_{\hc}\leqslant \sqrt{2}|p|_{\hc}|q|_{\hc}.
	\end{equation}
	Let $G\subset \mathbb{R}^3$ be a bounded domain. We consider functions of the form $u:G\rightarrow \hc$, $u=u_0+\sum_{k=1}^{3}u_ke_k$, where $u_k:G\rightarrow \mathbb{C}$, $k=0,1,2,3$. Let $\mathscr{F}(G)$ be a linear space of complex-valued functions (e.g., $C^k(G), L_p(G), W^{k,p}(G)$). We say that a biquaternionic-valued function $u$ belongs to the biquaternionic space $\mathscr{F}(G;\hc)$ if $u_k\in \mathscr{F}(G)$, $k=0,1,2,3$. For $1\leqslant p<\infty$, $L_p(G;\hc)$ is equipped with the norm $\|u\|_{L_p(G; \hc)}:=\left(\int_{\Omega}|u(x)|^p_{\hc}dV_x\right)^{\frac{1}{p}}$, and for $p=\infty$, $\|u\|_{L_{\infty}(G;\hc)}:= \operatorname{ess sup}\limits_{x\in G}|u(x)|_{\hc}$.  In particular, for the case $p=2$, $L_2(G; \hc)$ is a complex Hilbert space with the inner product $\langle u,v\rangle_{L_2(G;\hc)}:=\int_G \langle u(x),v(x)\rangle_{\hc}dV_x$. 
	
	Since $L_p(G; \hc)$ can be regarded as the product space $(L_p(G))^4$, and for $1\leqslant p<\infty$$, L_p(G)$ is separable \cite[Sec. 4.3]{brezis},  $L_p(G;\hc)$ is also separable. In particular, $L_2(G;\hc)$ is a separable complex Hilbert space. By the same argument, since $L_p(G)$ is reflexive for $1<p<\infty$ \cite[Th. 4.10]{brezis}, $L_p(G;\hc)$ is also reflexive.  The Sobolev spaces $W^{k,p}(G; \hc)$, with $k\in \mathbb{N}$ and $1\leqslant p\leqslant \infty$, the local spaces $L_{p,loc}(G; \hc)$, $W^{k,p}_{loc}(G; \hc)$, and $W_0^{k,p}(G; \hc)$ are defined in the usual way and equipped with corresponding norms. All these spaces are right $\hc$-modules. 
	
	The notation $\Omega\Subset G$ means that $\Omega$ is an open subset of $G$ with $\overline{\Omega}\subset G$. We say that $u:G\rightarrow \hc$ is {\it b-locally H\"older} in $G$ with H\"older exponent $0<\epsilon\leqslant 1$, if $u|_B\in C^{0,\epsilon}(\overline{B};\hc)$ for every ball $B\Subset G$, and we denote $u\in C_{bloc}^{0,\epsilon}(G;\hc)$.
	
	\begin{remark}\label{rem:sobolevembeddings}
	According to  \cite[Cor. 9.14]{brezis} and \cite[Th. 5. from Sec. 5.6 and Th. 4 from Sec. 5.8]{evans}, when $G$ is of class $C^1$, the following embeddings are compact: $W^{1,p}(G;\hc)\hookrightarrow L_q(G;\hc)$ for $1\leqslant p<3$ and $1\leqslant q<p^*:=\frac{3p}{3-p}$; $W^{1,3}(G;\hc)\hookrightarrow L_q(G;\hc)$ for $3\leqslant q<\infty$; and $W^{1,p}(G;\hc) \hookrightarrow C^{0,1-\frac{3}{p}}(\overline{G};\hc)$ for $3<p\leqslant \infty$. For a general open set $G$, we have the local embeddings: $W^{1,p}_{loc}(G;\hc)\hookrightarrow L_{q,loc}(G;\hc)$ for $1\leqslant p<3$ and $1\leqslant q<p^*$; $W^{1,3}_{loc}(G;\hc)\hookrightarrow L_{q,loc}(G;\hc)$ for $3\leqslant q<\infty$; and $W^{1,p}_{loc}(G;\hc) \hookrightarrow C^{0,1-\frac{3}{p}}_{bloc}(\overline{G};\hc)$ for $3<p\leqslant \infty$
	\end{remark}
	
    \subsection{Monogenic functions and integral operators}\label{Subsmonogenicint}
	The {\it Moisil-Theodorescu} operator (also called {\it Cauchy-Riemann} or {\it Dirac} operator) is defined as
	\[
	\mathbf{D}:= e_1\frac{\partial}{\partial x_1}+e_2\frac{\partial}{\partial x_2}+e_3\frac{\partial}{\partial x_3}.
	\]
	Its action, on the left and on the right, over a function $u\in C^1(G;\hc)$ is given by
	\begin{equation}\label{eq:actionofdiracop}
		\mathbf{D}u=-\dive \vec{u}+\grad u_0+\rot \vec{u} \quad \mbox{and }\quad u\mathbf{D}= -\dive \vec{u}+\grad u_0-\rot \vec{u}.
	\end{equation}
The following relation holds:
\begin{equation}\label{eq:conjugateD}
	\overline{\mathbf{D}u}=-\overline{u}\mathbf{D}, \quad \forall u\in C^1(G;\hc).
\end{equation}
These relations can be extended for $u\in W^{1,p}(G;\hc)$. A function $u\in C^1(G;\hc)$ is called {\it (left) monogenic} (regular or hyperholomorphic) if $\mathbf{D}u=0$. The set of all monogenic functions on $G$ is denoted by $\mathfrak{M}(G;\hc)$. From the identity \cite[Remark 1.2.6]{gurlebeckquater} $\bigtriangleup_3u=-\mathbf{D}^2u$ for $u\in C^2(G;\hc)$, we have that $\mathfrak{M}(G;\hc)\subset \operatorname{Har}(G;\hc) \subset C^{\infty}(G;\hc)$. In particular, $\{u_k\}_{k=0}^{3}\subset \operatorname{Har}(G)$. We recall the quaternionic Leibniz rule \cite[Th. 1.3.2]{gurlebeckquater}:
\begin{equation}\label{eq:leibnizrule}
	\mathbf{D}(uv)=(\mathbf{D}u)v+\overline{u}(\mathbf{D}v)-2\sum_{k=1}^{3}u_k\frac{\partial v}{\partial x_k}, \quad \forall u,v\in C^1(G;\hc).
\end{equation}
From \eqref{eq:leibnizrule}, $\mathfrak{M}(G,\hc)$ is a right $\hc$-module. Let $\Gamma=\partial G$ and suppose that $G$ and $\mathbb{R}^3\setminus \overline{G}$ are both connected. The {\it Cauchy kernel} is defined as
\begin{equation}\label{eq:Cauchykernel}
	E(x):=-\frac{\vec{x}}{4\pi |\vec{x}|^3}, \quad x\in \mathbb{R}^3\setminus \{0\}.
\end{equation}
The {\it Theodorescu operator} is given by \cite[Sec. 7.2]{gurlebeckhabetaspro}
\begin{equation}\label{eq:theodorescutransform}
	\mathbf{T}_{G}u(x):= -\int_G E(y-x)u(y)dV_y, \quad x\in G.
\end{equation}
Suppose that $\Gamma$ is a Liapunov surface, which implies that its exterior unit normal $\hat{\nu}(y)$ satisfies a H\"older condition over $\Gamma$ \cite[Sec. 27]{vladimirov}. In this case, the existence of the trace spaces $W ^{1-\frac{1}{p},p}(\Gamma; \hc)$ and the density of $C^{\infty}(\overline{G};\hc)$ in $W^{1,p}(G;\hc)$ for $1<p<\infty$, are guaranteed \cite[Th. 3.29 and 3.37]{maclean}. Consider $\hat{\nu}(y)$ as a purely vectorial quaternion. Gauss's theorem can be reformulated in terms of $\mathbf{D}$ as follows \cite[Prop. 3.22]{gurlebeckclifford}:
\begin{equation}\label{eq:Gausstheorem}
	\int_{\Gamma}\operatorname{tr}_{\Gamma}u(y)\hat{\nu}(y)\operatorname{tr}_{\Gamma}v(y)dS_y=\int_G\left((u\mathbf{D})(x)v(x)+u(x)(\mathbf{D}v(x))\right)dV_x,
\end{equation}
for all $u\in W^{1,p}(G;\hc), \; v\in W^{1,p'}(G;\hc)$, where $p'=\frac{p}{p-1}$. The (regular) {\it Cauchy operator} over $\Gamma$ is defined as \cite[Sec. 7.2]{gurlebeckhabetaspro}
\begin{equation}\label{eq:cauchyop}
	\mathbf{C}_{\Gamma}u(x):= \int_{\Gamma}E(y-x)\hat{\nu}(y)u(y)dS_y, \quad x\in \mathbb{R}^3\setminus \Gamma.
\end{equation} 
For $u\in C^{\infty}(\overline{G};\hc)$, the {\it Borel-Pompeiu formula} is valid \cite[Th. 1.4.4]{gurlebeckquater}:
\begin{equation}\label{eq:borelpompeiu}
	\mathbf{C}_{\Gamma}[u|_{\Gamma}](x)+\mathbf{T}_G[\mathbf{D}u](x)=\begin{cases}
		u(x), & \mbox{ if } x\in G,\\
		0, & \mbox{ if } x\in \mathbb{R}^3\setminus \overline{G}.
	\end{cases}
\end{equation}
The following proposition summarizes the properties of both integral operators, as found in \cite[Sec. 2.3 and 2.5]{gurlebeckquater},  \cite[Sec. 3.1]{gurlebeckclifford}, and \cite[Ch. 7]{gurlebeckhabetaspro}

\begin{proposition}\label{prop:integralop}
	Let $1<p<\infty$. The following statements hold.
	\begin{itemize}
		\item[(i)] $\mathbf{T}_G\in \mathcal{B}(L_p(G;\hc),W^{1,p}(G;\hc))$. Furthermore, $\mathbf{D}\mathbf{T}_Gu=u$ for all $u\in L_p(G;\hc)$.
		\item[(ii)] $\mathbf{C}_{\Gamma}\in \mathcal{B}(W^{1-\frac{1}{p},p}(\Gamma;\hc),W^{1,p}(G;\hc))$, and $\mathbf{C}_{\Gamma}\psi \in \mathfrak{M}(\mathbb{R}^3\setminus \Gamma;\hc)$ for all $\psi\in W^{1-\frac{1}{p},p}(\Gamma;\hc)$.
		\item[(iii)] The Borel-Pompeiu formula \eqref{eq:borelpompeiu} holds for $u\in W^{1,p}(G;\hc)$. In particular, if $u\in W^{1,p}_0(G;\hc)$, then $u=\mathbf{T}_G\mathbf{D}u$. 
		\item[(iv)] For $p>3$, $\mathbf{T}_{G}\in \mathcal{B}(L_p(G;\hc), C^{0,1-\frac{3}{p}}(\overline{G};\hc))$. 
	\end{itemize}
\end{proposition}

Using \eqref{eq:Gausstheorem} and \eqref{eq:conjugateD}, we can formulate a notion of weak $\mathbf{D}$-derivative\footnote{This definition was originally presented in \cite[Def. 2.4.1]{gurlebeckquater}, but with a mistake in the sign.}. Let $u\in L_{1,loc}(G;\hc)$. A function $v\in L_{1,loc}(G;\hc)$ is called the {\it weak} $\mathbf{D}$-derivative of $u$, if the following relation holds:
\begin{equation}\label{eq:defofweakD}
	\int_G \overline{u}(\mathbf{D}\varphi) =\int_G \overline{v}\varphi, \quad \forall \varphi\in C_0^{\infty}(G;\hc).
\end{equation}
In such a case, $v$ is unique a.e. in $G$, and we denote $v=\mathbf{D}_wu$. The class of functions in $L_p(G;\hc)$ possessing a weak $\mathbf{D}$-derivative in $L_p(G;\hc)$ is denoted by $\mathfrak{D}_p(G;\hc)$. 
\begin{remark}\label{remark:weakderi}
	\begin{itemize}
		\item[(i)] If $u\in W^{1,p}(G;\hc)$, then $\mathbf{D}u=\mathbf{D}_wu$. This follows from \eqref{eq:Gausstheorem}.
		\item[(ii)] $\mathfrak{D}_p(G;\hc)$ is a Banach space with the norm $\|u\|_{\mathfrak{D}_p(G;\hc)}:=\left( \|u\|_{L_p(G;\hc)}^p+\|\mathfrak{D}_wu\|_{L_p(G;\hc)}\right)^{\frac{1}{p}}$. The proof is similar to that of the completeness of Sobolev spaces.
	\item[(iii)] For every $u\in L_p(G;\hc)$, $\mathbf{D}_w\mathbf{T}_Gu=\mathbf{D}\mathbf{T}_Gu=u$. This follows from Proposition \ref{prop:integralop}(i) and point (i).
	\item[(iv)] Let $\mathfrak{M}^p(G;\hc):=\{u\in \mathfrak{D}_p(G;\hc)\, |\, \mathbf{D}_wu=0\}$. Take a scalar function $\varphi\in C_0^{\infty}(G)$ and $\Phi=\mathbf{D}\varphi$. By Definition \ref{eq:defofweakD} and the factorization $-\bigtriangleup_3=-\mathbf{D}^2$ we have:
	\[
	0=\int_G \overline{\mathbf{D}_{w}u}\mathbf{D}\Phi =-\int_G \overline{u}\bigtriangleup_3 \varphi,
	\]
	which implies that the components $\{u_k\}_{k=0}^{3}$ are weakly harmonic. It follows from the Weyl lemma \cite[Subsec. 7 of Sec. 21]{vladimirov} that $u\in C^{\infty}(G;\hc)$. Therefore $\mathfrak{M}^p(G;\hc) =\mathfrak{M}(G;\hc) \cap L_p(G;\hc)$.
	\item[(v)] When $u\in \mathfrak{D}_p(G;\hc)$ with $v=\mathbf{D}_wu$, point (iii) implies that $u=h+\mathbf{T}_Gv$, where $h\in \mathfrak{M}^p(G;\hc)$. From the previous point and Proposition \ref{prop:integralop}(i), $u\in W^{1,p}_{loc}(G;\hc)$.  Due to point (i), we use the notation $\mathbf{D}u$ for the weak $\mathbf{D}$-derivative.
	\end{itemize}
\end{remark}
The following lemma is based on ideas from \cite[Th. 3.1]{gurlebeckquater} and \cite[Th. 3.75 ]{gurlebeckclifford}. We include it for the sake of completeness.
\begin{lemma}\label{lema:integral}
	Let $1<p<\infty$ and $u\in W^{1,p}(G;\hc)$ such that
	\begin{equation}
		\int_{\Gamma}\psi(y)\hat{\nu}(y)\operatorname{tr}_{\Gamma}u(y)dS_y=0, \quad \forall \psi \in W^{1-\frac{1}{p'},p'}(G;\hc).
	\end{equation}
Hence $u=w+g$, where $w\in W_0^{1,p}(G;\hc)$ and $g\in W^{1,p}(G;\hc)\cap\mathfrak{M}^p(G;\hc)$. 
\end{lemma}
\begin{proof}
	Set $\varphi=\operatorname{tr}_{\Gamma}u$.  Consider the singular Cauchy operator\\ $\mathbf{S}_{\Gamma}\varphi(x)=2\operatorname{P.V.} \int_{\Gamma}E(y-x)\hat{\nu}(y)\varphi(y)dS_y$, $x\in \Gamma$. It is known that $\mathbf{S}_{\Gamma}\in \mathcal{B}\left(W^{1-\frac{1}{p},p}(\Gamma;\hc)\right)$ (see \cite[p. 421]{mkhlin}), and that the Plemelj-Sokhotski formulas are valid:
	\begin{equation}\label{eq:PS}
		\operatorname{tr}_{\Gamma^{\pm}}\mathbf{C}_{\Gamma}\varphi=\frac{1}{2}\left(\pm \varphi+\mathbf{S}_{\Gamma}\varphi\right),
	\end{equation}
where $\operatorname{tr}_{\Gamma^{\pm}}\mathbf{C}_{\Gamma}$ denotes the trace of $\mathbf{C}_{\Gamma}$ from $G$ and $\mathbb{R}^3\setminus \overline{G}$, respectively (see \cite[Remark 2.5.11]{gurlebeckquater}). Let  $\{q_n\}_{n=0}^{\infty}\subset \mathbb{R}^3\setminus\overline{G}$ be a dense set, and put $E_n(x)=E(x-q_n)$, $n\in \mathbb{N}_0$. Hence $\{E_n\}_{n=0}^{\infty}\subset W^{1,p'}(G;\hc)$. By hypothesis, $\int_{\Gamma}\operatorname{tr}_{\Gamma}E_n(y)\hat{\nu}(y)\operatorname{tr}_{\Gamma}u(y)dS_y=0$, but this is precisely $\mathbf{C}_{\Gamma}[\psi](q_n)=0$ for all $n\in \mathbb{N}_0$. Since $\{q_n\}_{n=0}^{\infty}$ is dense in $\mathbb{R}^3\setminus G$, by Proposition \ref{prop:integralop}(ii), $\mathbf{C}_{\Gamma}[\psi](x)=0$ for all $x\in \mathbb{R}^3\setminus \overline{G}$. Hence $\operatorname{tr}_{\Gamma^{-}}C_{\Gamma}\varphi=0$ and by \eqref{eq:PS}, $\varphi=\mathbf{S}_{\Gamma}\varphi$, which implies that $\operatorname{tr}_{\Gamma^{+}}C_{\Gamma}\varphi=\varphi$. Taking, $g=\mathbf{C}_{\Gamma}\varphi\in W^{1,p}(G;\hc)\cap\mathfrak{M}^p(G;\hc)$ and $w=u-g\in W_0^{1,p}(G;\hc)$, we obtain the result.
\end{proof}
\section{Properties of biquaternionic Vekua equation}
Given $p\in \hc$, denote $M^pq:= qp$ for all $q\in \hc$. Let $1<p<\infty$ and $A=(a_1,a_2,a_3,a_4)\in (L_{\infty}(G;\hc))^4$. Define the operator 
\begin{equation}\label{eq:opQ}
	\mathbf{Q}_Aw=M^{a_1}w+M^{a_2}\overline{w}+a_3w+a_4\overline{w}.
\end{equation}
Note that $\mathbf{Q}_A\in \mathcal{B}(L_p(G;\hc))$. We consider the Vekua equation
\begin{equation}\label{eq:Vekua1}
	\mathbf{D}w-\mathbf{Q}_Aw=0,\quad \mbox{ in } G.
\end{equation}
\begin{definition}
	The {\bf Vekua-Bergman space} associated with Eq. \eqref{eq:Vekua1}, is the class of all weak solutions $W\in L_p(G;\hc)$ of \eqref{eq:Vekua1}, denoted by $\mathfrak{V}_A^p(G;\hc)$.
\end{definition}
In the case $a_j\equiv 0$, $j=1,2,3,4$, $\mathfrak{V}_A^p(G;\hc)=\mathfrak{M}^p(G;\hc)$. We introduce the unbounded operator $\mathbf{D}-\mathbf{Q}_A:\mathfrak{D}_p(G;\hc)\subset L_p(G;\hc) \rightarrow L_p(G;\hc)$, together with the operator
\begin{equation}\label{eq:opS}
	\mathbf{S}_G^Aw:= w-\mathbf{T}_G\mathbf{Q}_Aw \quad \mbox{for } w\in L_p(G).
\end{equation}
By Proposition \ref{prop:integralop}(i), $\mathbf{S}_A\in \mathcal{B}(L_p(G))$. When $\frac{3}{2}<p<3$, note that $p^*=\frac{3p}{p-3}>3$. Thus, taking $3<q<p^*$, we have the following compositions: $L_p(G;\hc) \overset{\mathbf{T}_G}{\rightarrow} W^{1,p}(G)\hookrightarrow L_q(G;\hc) \hookrightarrow L_p(G;\hc)$. By Remark \ref{rem:sobolevembeddings}, $\mathbf{T}_G\in \mathcal{K}(L_p(G;\hc))$. Hence, $\mathbf{S}_G^A$ is a Fredholm operator with index $0$, and in particular, it possesses a finite-dimensional kernel \cite[Th. 2.22]{maclean}.

\begin{proposition}\label{prop:differentiabilityofsol}
	The following statements hold.
	\begin{itemize}
		\item[(i)] $\mathbf{S}_G^A(\mathfrak{D}_p(G;\hc)) \subset \mathfrak{D}_p(G;\hc)$, and 
		\begin{equation}\label{eq:relationSD}
			\mathbf{D}\mathbf{S}_G^Aw=(\mathbf{D}-\mathbf{Q}_A)w,\quad \forall w\in \mathfrak{D}_p(G;\hc).
		\end{equation}
	\item[(ii)] If $w\in \mathfrak{D}_p(G;\hc)$ is a weak solution of the non-homogeneous equation $(\mathbf{D}-\mathbf{Q}_A)w=v$, with $v\in L_p(G;\hc)$, then $w\in W^{1,p}_{loc}(G;\hc)$. In particular, $\mathfrak{V}_A^p(G;\hc)\subset W^{1,p}_{loc}(G;\hc)$. 
	\end{itemize}
\end{proposition}
\begin{proof}
	The contention $\mathbf{S}_G^A(\mathfrak{D}_p(G;\hc)) \subset \mathfrak{D}_p(G;\hc)$ follows from Remark \ref{remark:weakderi}(iii), and equality \eqref{eq:relationSD} follows from Proposition \ref{prop:integralop}(i). For the statement (ii), let $h=\mathbf{S}_G^Aw$. By \eqref{eq:relationSD}, $(\mathbf{D}-\mathbf{Q}_A)w=v$ iff $\mathbf{D}h=v$. Remark \ref{remark:weakderi}(v) implies $h=g+\mathbf{T}_Gv$ with $g\in \mathfrak{M}^p(G;\hc)$, from which we obtain $w=g+\mathbf{T}_G[\mathbf{Q}_Aw+v]$. Since $\mathbf{Q}_Aw+v\in L_p(G;\hc)$ and $g\in \mathfrak{M}^p(G;\hc)$, by Proposition \ref{prop:integralop}(i) and Remark \ref{remark:weakderi}(iv), we conclude that $w\in W^{1,p}_{loc}(G;\hc)$.
\end{proof}

\begin{corollary}\label{cor:vekuaiffmonogenic}
	$w\in \mathfrak{V}_A^p(G;\hc) $ iff $\mathbf{S}_G^Aw\in \mathfrak{M}^p(G;\hc)$.
\end{corollary}
\begin{proposition}\label{prop:operadorDQisclosed}
	The operator $\mathbf{D}-\mathbf{Q}_A: \mathfrak{D}_p(G;\hc)\rightarrow L_p(G;\hc)$ is closed.
\end{proposition}
\begin{proof}
	The case when $A\equiv (0,0,0,0)$ is given by Remark \ref{remark:weakderi}(ii). For the general case, let $\{w_n\}_{n=0}^{\infty}\subset \mathfrak{D}_p(G;\hc)$ and $w,v\in L_p(G;\hc)$ such that $w_n\rightarrow w$ and  $(\mathbf{D}-\mathbf{Q}_A)w_n\rightarrow v$ in $L_p(G;\hc)$. Taking $h_n=\mathbf{S}_G^Aw_n,h=\mathbf{S}_G^Aw\in \mathfrak{D}_p(G;\hc)$, $n\in \mathbb{N}_0$,  the continuity of $\mathbf{S}_G^A$ and relation \eqref{eq:relationSD} imply that $h_n\rightarrow h$ and $\mathbf{D}h_n\rightarrow v$ in $L_p(G;\hc)$. Since $\mathfrak{D}$ is closed in $L_p(G;\hc)$, $h\in \mathfrak{D}_p(G;\hc)$ and $\mathbf{D}h=v$. By the relations $w=h+\mathbf{T}_G\mathbf{Q}_Aw\in \mathfrak{D}_p(G;\hc)$ and \eqref{eq:relationSD}, we obtain that $\mathbf{D}-\mathbf{Q}_A$ is closed.
\end{proof}
\begin{theorem}
	The Vekua-Bergman space $\mathfrak{V}_A^p(G;\hc)$ is a separable and reflexive Banach space. In particular, $\mathfrak{V}_{A}^2(G;\mathbb{B})$ is a separable Hilbert space.
\end{theorem}
\begin{proof}
	The space $\mathfrak{V}_A^p(G;\hc)$ is closed in $L_p(G;\hc)$ (and consequently, it is a Banach space) since it is the null space of the closed operator $\mathbf{D}-\mathbf{Q}_A$. Since for $1<p<\infty$, the space $L_p(G;\hc)$ is separable and reflexive, hence $\mathfrak{V}_A^p(G;\hc)$ is also separable and reflexive (see \cite{brezis}, Propositions 3.21 and 3.25). 
\end{proof}

\section{Regularity of the solutions}

In this section, we focus on analyzing the regularity of the solutions of 
\begin{equation}\label{eq:Vekuanonhomo}
	(\mathbf{D}-\mathbf{Q}_A)w=v,\quad \mbox{ with } v\in L_p(G;\hc), \; \frac{3}{2}<p<\infty.
\end{equation}
\begin{proposition}\label{prop:estimatenorm}
	Let $w\in \mathfrak{D}_p(G;\hc)$ be a solution of \eqref{eq:Vekuanonhomo}. For every $D\Subset G$, there exists a constant $C_D>0$ (which does not depend on $w$ or $v$) such that
	\begin{equation}\label{ineq:estimatenorm1}
		\|w|_D\|_{W^{1,p}(D;\hc)}\leqslant C_D\left(\|v\|_{L_p(G;\hc)}+\|w\|_{L_p(G;\hc)}\right).
	\end{equation}
\end{proposition}
\begin{proof}
	Let $U$ another open set with $D\Subset U\Subset G$, and $\eta \in C_0^{\infty}(G)$ be a scalar cut-off function satisfying: $\eta \equiv 1$ in $D$ and $\operatorname{supp}\eta \subset U$. By Proposition \ref{prop:differentiabilityofsol}, $u=w\eta \in W_0^{1,p}(G)$. Thus, by formulas \ref{eq:borelpompeiu} and \ref{eq:leibnizrule}, the following relation holds a.e. in $G$:
	\begin{align*}
		u & =\mathbf{T}_G\mathbf{D}(w\eta)=\mathbf{T}_D\left[(\mathbf{D}u)\eta +\overline{u}\mathbf{D}\eta -2\sum_{k=1}^{3}u\frac{\partial \eta}{\partial x_k}\right]\\
		&= \mathbf{T}_D\left[(\mathbf{Q}_Aw+v)\eta +\overline{u}\mathbf{D}\eta -2\sum_{k=1}^{3}u\frac{\partial \eta}{\partial x_k}\right]
	\end{align*}
	In particular, for a.e. $x\in D$, $u=w$. Taking $M_G=\|\mathbf{T}\|_{\mathcal{B}(L_p(G;\hc),W^{1,p}(G;\hc) }$, we get
	\begin{align*}
		\|w\|_{W^{1,p}(D;\hc)}  \leqslant & M_G\Big{(}\left(\|\mathbf{Q}_A\|_{\mathcal{B}(L_p(G;\hc))}\|w\|_{L_p(G;\hc)}+\|v\|_{L_p(G;\hc)}\right) \|\eta\|_{L_{\infty}(G)} \\
	&	+ \sqrt{2}\|w\|_{L_p(G;\hc)} \|\mathbf{D}\eta\|_{L_{\infty}(G;\hc)}+2\|w\|_{L_p(G;\hc)}\|\eta\|_{W^{1,\infty}(G)}\Big{)}
	\end{align*}
Thus, we obtain \eqref{ineq:estimatenorm1} taking $C_D=M_G\max\{\|\mathbf{Q}_A\|_{\mathcal{B}(L_p(G;\hc))},2\}\|\eta\|_{W^{1,\infty}(G)}$.
\end{proof}
\begin{proposition}
	Let $w\in \mathfrak{V}_A^p(G;\hc)$ and $v\in L_q(G;\hc)$ satisfying \eqref{eq:Vekuanonhomo}, where $q\geqslant p$ and $q>3$. The following statements hold.
	\begin{itemize}
		\item[(i)] If $p<3$, then $w\in C^{0,1-\frac{3}{r}}_{bloc}(G;\hc)$ for all $3<r<\min\{p^*,q\}$.
		\item[(ii)] If $p\geqslant 3$, then  $w\in C^{0,1-\frac{3}{r}}_{bloc}(G;\hc)$ for all $3<r\leqslant q$.
	\end{itemize}
\end{proposition}
\begin{proof}
	Let $B\Subset G$ be a ball, and take another ball $B'$ with $B\Subset B'\Subset G$. By Proposition \ref{prop:differentiabilityofsol}(ii), $w|_{B'}\in W^{1,p}(B')$. Additionally, since $w|_{B'}$ is a solution of \eqref{eq:Vekuanonhomo} in $B'$, by relation \eqref{eq:relationSD} and Remark \ref{remark:weakderi}(v), the restriction $w|_{B'}$ has the form
	\begin{equation}\label{eq:aux1}
		w|_{B'}=g_{B'}+\mathbf{T}_{B'}\left(\mathbf{Q}_Aw|_{B'}+v|_{B'}\right), \quad \mbox{ where }\; g_{B'}\in \mathfrak{M}^p(B';\hc).
	\end{equation}
Since $g_{B'}\in C^{\infty}(B';\hc)$, it remains only to analyze the integral $\mathbf{T}_{B'}\left(\mathbf{Q}_Aw|_{B'}+v|_{B'}\right)$.
\begin{itemize}
	\item[(i)]   Since $p>\frac{3}{2}$, we have $p^*>3$. Remark \ref{rem:sobolevembeddings} implies that $\mathbf{Q}_Aw|_{B'}+v|_{B'}\in L_r(B';\hc)$ for $3<r<\min\{p^*,q\}$. By Proposition \ref{prop:integralop}(iv), $\mathbf{T}_{B'}\left(\mathbf{Q}_Aw|_{B'}+v|_{B'}\right)\in C^{0,1-\frac{3}{r}}(\overline{B'};\hc)$. Hence $w|_{B}=(w|_{B'})|_B\in C^{0,1-\frac{3}{r}}(\overline{B};\hc)$. Since $B$ was arbitrary, we conclude (i).
	\item[(ii)] In this case, Remark \ref{rem:sobolevembeddings} implies $\mathbf{Q}_Aw|_{B'}+v|_{B'}\in L_r(B';\mathbb{B})$ for $3<r\leqslant q$. According to Proposition \ref{prop:integralop}(iv), $\mathbf{T}_{B'}\left(\mathbf{Q}_Aw|_{B'}+v|_{B'}\right)\in C^{0,1-\frac{3}{r}}(\overline{B'};\hc)$, and hence $w|_{B}=(w|_{B'})|_B\in C^{0,1-\frac{3}{r}}(\overline{B};\hc)$.  Due to the arbitrariness of $B$, we conclude (ii).
\end{itemize}
\end{proof}

Taking $v\equiv 0$ we obtain the regularity for the solutions of \eqref{eq:Vekua1}.
\begin{corollary}\label{coro:regularityofsol}
	For $\frac{3}{2}<p<\infty$, the following statements hold.
	\begin{itemize}
		\item[(i)] If $p<3$, then $\mathfrak{V}_A^p(G;\hc) \subset C^{0,1-\frac{3}{r}}_{bloc}(G;\hc)$ for all $3<r<p^*$.
		\item[(ii)] If $p\geqslant 3$, then $\mathfrak{V}_A^p(G;\hc) \subset C^{0,1-\frac{3}{r}}_{bloc}(G;\hc)$ for all $r>3$.
	\end{itemize}
\end{corollary}

In particular, $\mathfrak{V}_A^p(G;\hc)\subset C(G;\hc)$ and for every $x\in G$, the evaluation map\\ $\mathfrak{V}_A^p(G;\hc)\ni w\mapsto w(x)\in \hc$ is well defined.

\begin{remark}\label{remark:bergmanspacemono}
	According to \cite[Prop. 2.2]{gurlebeckquater}, for every compact $K\subset G$, there exists $C_K>0$ such that
	\begin{equation*}
		\max_{x\in K}|u(x)|_{\hc}\leqslant C_K\|u\|_{L_p(G;\hc)},\quad \forall u\in \mathfrak{M}^p(G;\hc).
	\end{equation*}
This property generalizes the classical analytic Bergman spaces to monogenic functions.
\end{remark}

\begin{theorem}\label{th:boundednessofevalfunct}
	Suppose $\frac{3}{2}<p<\infty$. For every compact $K\subset G$, there exists $C_K>0$ such that
	\begin{equation}\label{ineq:boundednessofevalfunct}
		\max_{x\in K}|w(x)|_{\hc}\leqslant C_K\|w\|_{L_p(G;\hc)},\quad \forall w\in \mathfrak{V}_A^p(G;\hc).
	\end{equation}
\end{theorem}
\begin{proof}
	First, consider the case when $K=\overline{B}$, with $B\subset G$ being a ball. Consider another ball $B'$ with $B\Subset B'\Subset G$, and write $w|_{B'}=u+\mathbf{T}_{B'}\mathbf{Q}_Aw|_{B'}$, where $u=\mathbf{S}_{B'}^Aw|_{B'}\in \mathfrak{M}^p(B';\hc)$ (by Corollary \ref{cor:vekuaiffmonogenic}). By Remark  \ref{remark:bergmanspacemono}, there is $C_{\overline{B}}^1$ such that
	\[
	\max_{x\in \overline{B}}|u(x)|_{\hc}\leqslant C_{\overline{B}}^1\|u\|_{L_p(B')}\leqslant C_{\overline{B}}^1\|\mathbf{S}_{B'}^A\|_{\mathcal{B}(L_p(B';\hc))}\|w\|_{L_p(G;\hc)}.
	\]
	Now, by Remark \ref{rem:sobolevembeddings} and Corollary \ref{coro:regularityofsol}, there exists $r>3$ such that $w|_{B'}\in L_r(B';\hc)$, and the embeddings
	\[
	W^{1,p}(B';\hc) \hookrightarrow L_r(B';\hc) \quad \mbox{and  } W^{1,r}(B';\hc) \hookrightarrow C^{0,1-\frac{3}{r}}(\overline{B'};\hc)
	\]
	are compact. Let $C_{B'}^1$ and $C_{B'}^2$ be their corresponding norms. By Proposition \ref{prop:estimatenorm}, there exists a constant $C_{B'}^3>0$ (independent of $w$) such that $\|w|_{B'}\|_{W^{1,p}(B';\hc)}\leqslant C_{B'}^3\|w\|_{L_p(G;\hc)}$. Hence, we have that  $v=\mathbf{T}_{B'}\mathbf{Q}_Aw|_{B'}\in C^{0,1-\frac{3}{r}}(\overline{B'};\hc)$. Taking $M_{B'}=\|\mathbf{T}_{B'}\mathbf{Q}_A\|_{\mathcal{B}(L_r(G;\hc)),W^{1,r}(G;\hc)}$, we obtain
	\begin{align*}
		\max_{x\in \overline{B}}|v(x)|_{\hc} & \leqslant \|v\|_{C^{0,1-\frac{3}{r}}(B';\hc)} \leqslant C_{B'}^2\|v\|_{W^{1,r}(B',\hc)}\leqslant C_{B'}^2M_{B'}\|w|_{B'}\|_{L_r(B';\hc)}\\
		&\leqslant C_{B'}^2M_{B'}C_{B'}^1\|w|_{B'}\|_{W^{1,p}(B';\hc)}\leqslant C_{B'}^2M_{B'}C_{B'}^1C_{B'}^3\|w\|_{L_p(G;\hc)}.
	\end{align*}
Choosing $C_{\overline{B}}=\max\{C_{\overline{B}}^1\|\mathbf{S}_{B'}^A\|_{\mathcal{B}(L_p(B';\hc))},C_{B'}^2M_{B'}C_{B'}^1C_{B'}^3\}$, we conclude that $\max\limits_{x\in \overline{B}}|w(x)|_{\hc}\leqslant C_{\overline{B}}\|w\|_{L_p(G;\hc)}$. For a general compact subset $K$, take $B_1,\dots, B_N$ open balls with $K\subset \bigcup_{j=1}^N\overline{B}_j\subset G$. Therefore, inequality \eqref{ineq:boundednessofevalfunct} is satisfied by taking $C_K=\displaystyle \max_{1\leqslant j\leqslant N}C_{\overline{B}_j}$.
\end{proof}

\section{Existence of the reproducing kernel}

In this section, we focus on the case $p=2$. Theorem \ref{th:boundednessofevalfunct} implies that for each $k=0,1,2,3$, and every $x\in G$, the functional $\mathfrak{V}_A^2(G;\hc)\ni w\mapsto w_k(x)\in \mathbb{C}$ is bounded. By the Riesz representation theorem, there exists $K_x^k\in \mathfrak{V}_A^2(G;\hc)$ satisfying
\begin{equation}\label{eq:reproker1}
	w_k(x)=\langle K_x^k,w\rangle_{L_2(G;\hc)}, \quad \forall w\in \mathfrak{V}_A^2(G;\hc).
\end{equation}
Let us denote
\begin{equation}\label{eq:Kcomponents}
	K_x^k(t)=\sum_{j=0}^3K_x^{k,j}(t)e_j,
\end{equation}
  where $K_x^{j,k}\in L_2(G)$ are the component functions of $K_x^k$.

\begin{remark}\label{remark:matrixK}
	For $k,j=0,1,2, 3$ and $x,t\in G$, relation \eqref{eq:reproker1} implies
	\begin{equation}\label{eq:Khermitian}
			K_x^{k,j}(t)= \langle K_t^{j},K_x^{k}\rangle_{L_2(G;\hc)}=\langle K_x^{k},K_t^{j}\rangle_{L_2(G;\hc)}^*= (K_t^{j,k}(x))^*.
	\end{equation}
	When $x=t$, $k=j$, we obtain $K^{k,k}_x(x)=\|K_x^k\|_{L_2(G;\hc)}^2>0$. In particular, the matrix $(K_x^{k,j}(x))_{j,k=0}^{3}$ is Hermitian with a positive diagonal.
\end{remark}

Given $w\in \mathfrak{V}_A^2(G;\hc)$, we have the representation $w(x)=\sum_{k=0}^{3}\langle K_x^{k},w\rangle_{L_2(G;\hc)}e_k$. Developing the inner product and using \eqref{eq:Khermitian} we obtain
\begin{align*}
	w(x) =& \sum_{k=0}^{3}\left[\int_G\left(\sum_{j=0}^3(K_x^{k,j}(t))^*w_j(t)\right)dV_t\right] e_k = \int_G\left(\sum_{k=0}^{3}\sum_{j=0}^{3}K_t^{j,k}(x)w_j(t)e_k\right)dV_t\\
	=& \int_G\left[\sum_{j=0}^{3}\left(\sum_{k=0}^{3}K_t^{j,k}(x)e_k\right)w_j(t)\right]dV_t=\int_G\left(\sum_{j=0}^{3}K_t^j(x)w_j(t)\right)dV_t,
\end{align*}
where the last equality follows from \eqref{eq:Kcomponents}. Following  \cite{camposbergman,mineVekua2}, we introduce the  definition for the Bergman kernel.
\begin{definition}\label{def:Bergmankernel}
	The {\bf Vekua-Bergman kernel} for the space $\mathfrak{V}_A^2(G;\hc)$ with coefficient $a\in \hc$ is defined as
	\begin{equation}\label{eq:Bergmankernel}
		\mathscr{K}_G^A(x,t;a) := \sum_{j=0}^3K_t^{j}(x)a_j, \quad x,t\in G.
	\end{equation}
\end{definition}
\begin{remark}
	\begin{itemize}
		\item[(i)] $\mathscr{K}_G^A(\cdot,t;a)\in \mathfrak{V}_A^2(G;\hc)$ for all $t\in G$, $a\in \hc$.
		\item[(ii)] The following reproducing property holds
         \begin{equation}\label{eq:reprodproeprtyvekua}
         	w(x)=\int_G\mathscr{K}_G^A(x,t;w(t))dV_t, \quad \forall w\in \mathfrak{V}_A^2(G;\hc), \; x\in G.
         \end{equation}	
	\end{itemize}
\end{remark}
Since $\mathfrak{V}_A^2(G;\hc)$ is separable, it admits a countable orthonormal basis $\{\varphi_n\}_{n=0}^{\infty}$ \cite[Ch. I, Prop. 4.16]{conwayfunctional}. Since $\mathscr{K}_G^A(\cdot,t;a)\in \mathfrak{V}_A^2(G;\hc)$ for all $x\in G$, $a\in \hc$, it possesses a Fourier series in terms of $\{\varphi_n\}_{n=0}^{\infty}$, and the form of its Fourier coefficients can be obtained in terms of the coefficient $a$.

\begin{proposition}
	For every $t\in G$, $a\in \hc$, the kernel $\mathscr{K}_G^A(x,t;a)$ admits the Fourier series expansion
	\begin{equation}\label{eq:bergmankernelexpansion}
		\mathscr{K}_G^A(x,t;a)\sum_{n=0}^{\infty}\langle  \varphi_n(t), a\rangle_{\hc}\varphi_n(x).
	\end{equation}
	The series converges in the variable $x$ in the $L_2(G;\hc)$-norm and uniformly on compact subsets of $G$.
\end{proposition}
\begin{proof}
	The Fourier series of $\mathscr{K}_G^A(x,t;a)$ is given by $\mathscr{K}_G^A(x,t;a)=\sum_{n=0}^{\infty}\langle \varphi_n,\mathscr{K}_G^A(\cdot,t;a)\rangle_{L_2(G;\hc)}\varphi_n(x)$ and converges in $x$ with respect to the $L_2(G;\hc)$-norm. The Fourier coefficients have the form
	\begin{align*}
		\langle \varphi_n, \mathscr{K}_G^A(\cdot,t;a)\rangle_{L_2(G;\hc)}  &= \left\langle \varphi_n,\sum_{k=0}^{3}K_t^ka_k\right\rangle_{L_2(G;\hc)}=\sum_{k=0}^{3}\langle \varphi_n,K_t^k\rangle_{L_2(G;\hc)}a_k\\
		&= \sum_{k=0}^{3}\varphi_{n,k}^*(t)a_k=\langle \varphi_n(t),a\rangle_{\hc},
	\end{align*}
from which we obtain \eqref{eq:bergmankernelexpansion}. The uniform convergence on compact subsets of $G$ is due to Theorem \ref{th:boundednessofevalfunct}.
\end{proof}

Since $\mathfrak{V}_A^2(G;\hc)$ is a closed subspace of $L_2(G;\hc)$, there exists the bounded orthogonal projection of $L_2(G;\hc)$ onto $\mathfrak{V}_A^2(G;\hc)$ \cite[Ch. I, Th. 2.7]{conwayfunctional}, denoted as $\mathbf{P}_{G}^A$. We refer to it as the {\it Vekua-Bergman projection}. Similar to the case of the  bicomplex Vekua equation \cite[Prop. 26]{mineVekua2}, the Vekua-Bergman projection can be expressed in terms of the kernel $\mathscr{K}_G^A(x,t;a)$. 

\begin{proposition}\label{prop:Bergmanproj}
	The Vekua-Bergman projection can be written as
	\begin{equation}\label{eq:Bergmanprojection}
		\mathbf{P}_G^Au(x)= \int_G\mathscr{K}_G^A(x,t;u(t))dV_t, \qquad \forall u \in L_2(G;\hc).
	\end{equation}
\end{proposition}
\begin{proof}
	Let $u \in L_2(G;\hc)$. A direct computation shows that
	\[
	\int_G\mathscr{K}_G^A(x,t;u(t))dV_t=\sum_{k=0}^{3}\langle K_x^k,u\rangle_{L_2(G;\hc)}e_k.
	\]
	Since $\mathbf{P}_G^Au\in \mathfrak{V}_A^2(G;\hc)$, using \eqref{eq:reproker1} and the fact that an orthogonal projection is a self-adjoin operator \cite[Prop. 3.3]{conwayfunctional}, we get
	\begin{align*}
		\mathbf{P}_G^Au(x) &= \sum_{k=0}^{3}\langle K_x^k, \mathbf{P}_G^Au\rangle_{L_2(G;\hc)}e_k=\sum_{k=0}^{3}\langle \mathbf{P}_G^AK_x^k, u\rangle_{L_2(G;\hc)}e_k\\
		& =\sum_{k=0}^{3}\langle K_x^k, u\rangle_{L_2(G;\hc)}e_k=\int_G\mathscr{K}_G^A(x,t;u(t))dV_t.
	\end{align*}
\end{proof}

The space $L_2(G;\hc)$ can be treated as a right $\hc$-Hilbert space with the $\hc$-inner product
\begin{equation}\label{eq:hcinnerproduct}
	\langle \langle u|v\rangle \rangle _{L_2(G;\hc)}=\int_G u^{\dagger}v, \quad u,v\in L_2(G;\hc).
\end{equation}
The product \eqref{eq:hcinnerproduct} is right $\hc$-linear in the second slot and satisfies $\langle \langle u|v\rangle \rangle_{L_2(G;\hc)}= \langle \langle v|u \rangle \rangle _{L_2(G;\hc)}^{\dagger}$. Hence $\langle \langle ua|v\rangle \rangle = a^{\dagger} \langle \langle u|v\rangle \rangle_{L_2(G;\hc)}$ for all $a\in \hc$. Note that $\langle u,v \rangle_{L_2(G;\hc)}= \sc \langle \langle u|v\rangle \rangle_{L_2(G;\hc)}$.

The Bergman space $\mathfrak{V}_A^2(G;\hc)$ is not always a right $\hc$-module. For example,  $\mathfrak{V}_{(0,0,a_3,0)}^2(G;\hc)$ is a right $\hc$-module, but even for $A=(\vec{a},0,0,0)$ with $\vec{a}$ a nonzero constant vector, $\mathfrak{V}_A^2(G;\hc)$ is not a right $\hc$-module. 

\begin{theorem}
	If $\mathfrak{V}_A^2(G;\hc)$ is a right $\hc$-module, then there exists a kernel $\mathscr{K}_G^A(x,t)$ with $\mathscr{K}_G^A(\cdot,t)\in \mathfrak{V}_A^2(G;\hc)$ for all $t\in G$, satisfying 
	\begin{equation}\label{eq:Bergmankernel2}
		w(x)=\int_G \mathscr{K}_G^A(x,t)w(t)dt,\quad \forall w\in \mathfrak{V}_A^2(G;\hc),\; x\in G.
	\end{equation}
Conversely, if such a kernel exists, then $\mathfrak{V}_A^2(G;\hc)$ is a right $\hc$-module.
\end{theorem}
\begin{proof}
	Fix $x\in G$ and let $w\in \mathfrak{V}^2_A(G;\hc)$ with component functions $\{w_k\}_{k=0}^{3}\subset L_2(G)$. Note that $\operatorname{Sc}(w(x)e_k)=-w_k(x)$, hence
	\[
	w(x)= w_0(x)-\sum_{k=1}^3\operatorname{Sc}(w(x)e_k)e_k.
	\]
Let $K_x^0(t)$ be the kernel satisfying $w_0(x)=\langle K_x^0,w\rangle_{L_2(G;\hc)}$. Using the equality $\overline{\langle \langle u|v\rangle \rangle_{L_2(G;\hc)}}=\langle \langle v|u\rangle \rangle_{L_2(G;\hc)}^*$, we obtain 
\begin{align*}
	w(x)  = & \langle K_x^0,w\rangle_{L_2(G;\hc)}-\sum_{k=1}^{3}\langle K_x^0,we_k\rangle_{L_2(G;\hc)}e_k\\
	 = & \frac{1}{2}\left(\langle\langle K_x^0|w\rangle\rangle_{L_2(G;\hc)}+\overline{\langle\langle K_x^0|w\rangle\rangle_{L_2(G;\hc)}} \right)\\
	& -\frac{1}{2}´\sum_{k=1}^{3}\left(\langle\langle K_x^0|we_k\rangle\rangle_{L_2(G;\hc)}+
	\overline{\langle\langle K_x^0|we_k\rangle\rangle_{L_2(G;\hc)}} \right)e_k\\
	= & \frac{1}{2}\left(\langle\langle K_x^0|w\rangle\rangle_{L_2(G;\hc)}+\langle\langle w|K_x^0\rangle\rangle_{L_2(G;\hc)}^* \right)\\
	& -\frac{1}{2}´\sum_{k=1}^{3}\left(\langle\langle K_x^0|we_k\rangle\rangle_{L_2(G;\hc)}+
	\langle\langle we_k|K_x^0\rangle\rangle_{L_2(G;\hc)}^* \right)e_k\\
	=& \frac{1}{2}\left(\langle\langle K_x^0|w\rangle\rangle_{L_2(G;\hc)}+\langle\langle w|K_x^0\rangle\rangle_{L_2(G;\hc)}^* \right)+\frac{1}{2}\sum_{k=1}^{3}\langle\langle K_x^0|w\rangle\rangle_{L_2(G;\hc)}\\
	& +\frac{1}{2}\sum_{k=1}^3e_k \langle\langle w|K_x^0\rangle\rangle_{L_2(G;\hc)}^*e_k\\
	=& 2\langle\langle K_x^0|w\rangle\rangle_{L_2(G;\hc)}+\frac{1}{2}\langle\langle w|K_x^0\rangle\rangle_{L_2(G;\hc)}^*+\frac{1}{2}\sum_{k=1}^3e_k \langle\langle w|K_x^0\rangle\rangle_{L_2(G;\hc)}^*e_k.
\end{align*}	
A direct computation shows that $\sum_{k=1}^{3}e_kae_k=-4a_0+a$ for every $a\in \hc$. Consequently, $a+\sum_{k=1}^{3}e_kae_k=-2\overline{a}$. Thus,
\begin{align*}
	w(x)= & 2\langle\langle K_x^0|w\rangle\rangle_{L_2(G;\hc)}-\overline{\langle\langle w|K_x^0\rangle\rangle_{L_2(G;\hc)}^*}\\
	=& 2\langle\langle K_x^0|w\rangle\rangle_{L_2(G;\hc)}-\langle\langle w|K_x^0\rangle\rangle_{L_2(G;\hc)}^{\dagger}\\
	=& \langle\langle K_x^0|w\rangle\rangle_{L_2(G;\hc)}.
\end{align*}
Note that $K_x^0(t)=\langle\langle K_t^0|K_x^0\rangle \rangle_{L_2(G;\hc)}=(K_t^0(x))^{\dagger}$. Thus, $\mathscr{K}_G^A(x,t)=(K_x^0(t))^{\dagger}=K_t^0(x)$ satisfies \eqref{eq:Bergmankernel2}. 

Reciprocally, suppose that $\mathscr{K}_G^A(x,t)$ satisfies \eqref{eq:Bergmankernel2}. Let $K_x^k(t)$ be such that $w_k(x)=\langle K_x^k,w\rangle_{L_2(G;\hc)}$, $k=0,1,2,3$.  Note that $\langle (\mathscr{K}_G^A(x,\cdot))^{\dagger},w\rangle_{L_2(G;\hc)}= w_0(x)=\langle K_x^0,w\rangle_{L_2(G;\hc)}$ for all $w\in \mathfrak{V}_A^2(G;\hc)$ and hence, $\mathscr{K}_G^A(x,t)=(K_x^0(t))^{\dagger}$. On the other hand,
\[
\langle K_x^0,we_k\rangle_{L_2(G;\hc)}= \sc (\langle\langle K_x^0|w\rangle\rangle_{L_2(G;\hc)}e_k)=-w_k(x)=-\langle K_x^k,w\rangle_{L_2(G;\hc)}.
\]
A direc computation shows that $\langle K_x^0,we_k\rangle_{L_2(G;\hc)}=-\langle K_x^0 e_k,w\rangle_{L_2(G;\hc)}$. Hence, $\langle K_x^0e_k,w\rangle_{L_2(G;\hc)}=\langle K_x^k,w\rangle_{L_2(G;\hc)}$ for all $w\in \mathfrak{V}_A^2(G;\hc)$, from which we conclude that $K_x^0(t)e_k=K_x^k(t)$ for all $x,t\in G$, $k=0,1,2,3$. Consequently,
\[
\mathscr{K}_G^A(x,t;a)=\sum_{k=0}^{3}K_t^k(x)a_k=K_t^0(x)a_1+\sum_{k=1}^{3}K_t^{0}(x)e_ka_k= K_t^0(x)a.
\]
Repeating the same procedure as in the first part, we find  $K_x^0(t)=(K_t^0(x))^{\dagger}$. Thus, from \eqref{eq:Bergmankernel} we obtain that 
\[
\mathscr{K}_G^A(x,t;a)= \mathscr{K}_G^A(x,t)a.
\]
In consequence, the orthogonal projection $\mathbf{P}_G^A$ has the form
\[
\mathbf{P}_G^Au(x)=\int_{G} \mathscr{K}_G^A(x,t)u(t)dV_t,\quad \forall u\in L_2(G;\hc).
\]
Finally, taking $w\in \mathfrak{V}_A^2(G;\hc)$ and $a\in \hc$, we get
\[
w(x)a=(\mathbf{P}_G^A[w](x))a=\int_{G}\mathscr{K}_G^A(x,t)w(t)adV_t=\mathbf{P}_G^A[Wa](x).
\]
Hence, $wa\in \mathfrak{V}_A^2(G;\hc)$. Therefore, $\mathfrak{V}_A^2(G;\hc)$ is a right $\hc$-module.
\end{proof}

Consequently, Bergman spaces associated with operators of the type $\mathbf{D}\pm M^{a_1(x)}$, $\mathbf{D}\pm M^{a_2(x)}\mathbf{C}_{\hc}$ and $\mathbf{D}\pm a_4(x)\mathbf{C}_{\hc}$, generally do not admit a kernel of type \eqref{eq:Bergmankernel2}.

\section{Annihilators and orthogonal decompositions}
Let $X$ be a Banach space and $X^{\star}$ its dual. The duality pairing between $X$ and $X^{\star}$ is usually denoted by $(\varphi|x)_X:=\varphi(x)$, $x\in X, \varphi\in X^{\star}$. Given a densely defined operator $\mathbf{A}:\operatorname{dom}(\mathbf{A})\subset X\rightarrow X$, its transpose is the operator $\mathbf{A}^T: \operatorname{dom}(\mathbf{A})\subset X^{\star} \rightarrow X^{\star}$, whose domain consists of the functionals $\varphi\in X^{\star}$ for which there exists $\psi\in X^{\star}$ satisfying $(\varphi|\mathbf{A}x)_X=(\psi|x)_X$ for all $x\in \operatorname{dom}(\mathbf{A})$ \cite[Sec. 2.6]{brezis}. Such $\psi$ is unique, and we define $\mathbf{A}^T\varphi=\psi$. The operator $\mathbf{A}^T$ is closed in $X^{\star}$ \cite[Prop. 2.17]{brezis}.

\begin{remark}\label{remark:transposesumandproduct}
	Let $\mathbf{A}:\operatorname{dom}(\mathbf{A})\subset X\rightarrow X$ be a densely defined operator, and $\mathbf{S}\in \mathcal{B}(X)$. The following relations hold:
	\[
	\mathbf{S}^T\mathbf{A}^T\subset (\mathbf{AS})^T\quad \mbox{and } (\mathbf{A}+\mathbf{S})^T=\mathbf{A}^T+\mathbf{S}^T.
	\]
	The proof is a straightforward.
\end{remark}

Using the Riesz-Fisher representation theorem, it is a straightforward to show that for $1<p<\infty$ and every $\Phi \in \left(L_p(G;\hc)\right)^{\star}$, there exists a unique (a.e. in $G$) $v\in L_{p'}(G;\hc)$ such that $\Phi(u)= \sc\int_G \overline{v}u$ for all $u\in L_p(G;\hc)$. The duality pairing between $L_{p'}(G;\hc)$ and $L_p(G;\hc)$ is denoted by $(v|u)_G := \sc \int_G\overline{v}u$. Consequently,  the transpose of a densely defined operator $\mathbf{A}: \operatorname{dom}(\mathbf{A})\subset L_p(G;\hc)\rightarrow L_p(G;\hc)$ can be realized as an operator $\mathbf{A}^T: \operatorname{dom}(\mathbf{A}^T) \subset L_{p'}(G;\hc) \rightarrow L_{p'}(G;\hc)$ whose domain $\operatorname{dom}(\mathbf{A}^T)$ consists of the elements $v\in L_{p'}(G;\hc)$ for which there exists  $w\in L_{p'}(G;\hc)$ satisfying $(v|\mathbf{A}u)_G=(w|u)_G$ for all $u\in \mathscr{D}(\mathbf{A})$, and its actions is given by $\mathbf{A}^Tv=w$. Such operator is closed in $L_{p'}(G;\hc)$.

\begin{proposition}\label{prop:adjointsofsomeop}Let $1<p<\infty$. The following statements hold.
	\begin{itemize}
		\item[(i)] $\mathbf{T}_G^T= \mathbf{T}_G$ (as an operator acting in $L_{p'}(G;\hc)$).
		\item[(ii)] If $A\in (L_{\infty}(G;\hc))^4$, then $\mathbf{Q}_A^T=\mathbf{Q}_{A^{\star}}$, where $A^{\star}:=(\overline{a_1},a_4,\overline{a_3},a_2)$.    
	\end{itemize}
\end{proposition}
\begin{proof}
	\begin{itemize}
		\item[(i)] Take $\varphi,\psi\in C_0^{\infty}(G;\hc)$. Consider
		\begin{align*}
			(\psi|\mathbf{T}_G\varphi)_G & = \int_{G}\overline{\psi(x)}\mathbf{T}_G\varphi(x)dV_x= \frac{1}{4\pi}\int_G\overline{\psi(x)}\left[\int_G\frac{\vec{y}-\vec{x}}{|\vec{y}-\vec{x}|^3}\varphi(y)dV_y\right]dV_x\\
			&= \frac{1}{4\pi}\int_{G}\left[\int_G\frac{\overline{\psi(x)}(\vec{y}-\vec{x})}{|\vec{y}-\vec{x}|^3}dV_x\right]\varphi(y)dV_y=\frac{1}{4\pi}\int_{G}\overline{\left[\int_G\frac{(\vec{x}-\vec{y})\psi(x)}{|\vec{x}-\vec{y}|^3}dV_x\right]}\varphi(y)dV_y
		\end{align*}
	From this, $(\psi|\mathbf{T}_G\varphi)_G=(\mathbf{T}_G\psi|\varphi)_G$. By the continuity of $\mathbf{T}_G$ and the density of $C_0^{\infty}(G;\hc)$ in $L_p(G;\hc)$ and $L_{p'}(G;\hc)$, we obtain $\mathbf{T}_G^T=\mathbf{T}_G:L_{p'}(G;\hc)\rightarrow L_{p'}(G;\hc)$.
	\item[(ii)] Let $a\in L_{\infty}(G;\hc)$ and $u\in L_p(G;\hc),v\in L_{p'}(G;\hc)$. Using the fact that $\sc(pq)=\sc(qp)=\sc(\overline{pq})$, we obtain
	\begin{align*}
		(v|ua)_G&=\sc \int_{G} \overline{v}(ua)= \sc \int_G a(\overline{v}u)=\sc \int_G \overline{(v\overline{a})}u=(v\overline{a}|v)_G\\
		(v|au)_G&=\sc \int_G \overline{v}(au)= \sc \int_G \overline{(\overline{a}v)}u=(\overline{a}v|u)_G,\\
		(v|\overline{u})_G&= \sc\int_G \overline{v}\;\overline{u}=\sc \int_G \overline{uv}= \sc \int_G uv =\sc \int_G vu= (\overline{v}|u)_G.
	\end{align*}
Consequently, $(M^a)^T=M^{\overline{a}}$, $(a\mathbf{I})^T=\overline{a}\mathbf{I}$, and $\mathbf{C}_{\hc}^T=\mathbf{C}$. In the same way, $(M^a\mathbf{C}_{\hc})^Tv=\mathbf{C}_{\hc}^T(M^a)^Tv=\mathbf{C}_{\hc}(v\overline{a})=a\overline{v}$ and $(a\mathbf{C}_{\hc})^Tv=\mathbf{C}_{\hc}(\overline{a}v)=\overline{v}a$. Thus, $(M^a\mathbf{C}_{\hc})^T=a\mathbf{C}_{\hc}$ and $(a\mathbf{C}_{\hc})^T=M^a\mathbf{C}_{\hc}$. From these relations, we conclude that
\[
\mathbf{Q}_A^T=(M^{a_1}+M^{a_2}\mathbf{C}_{\hc}+a_3\mathbf{I}+a_4\mathbf{C}_{\hc})^T=M^{\overline{a_1}}+a_2\mathbf{C}_{\hc}+\overline{a_3}\mathbf{I}+M^{a_4}\mathbf{C}_{\hc}=\mathbf{Q}_{A^{\star}},
\]
 where $A^{\star}=(\overline{a_1},a_4,\overline{a_3},a_2)$.
	\end{itemize}
\end{proof}
\begin{proposition}\label{prop:transposeD}
	Let $1<p<\infty$. The transpose of the Moisil-Theodorescu operator $\mathbf{D}: \mathfrak{D}_p(G;\hc)\subset L_p(G;\hc) \rightarrow L_p(G;\hc)$ is given by $\mathbf{D}: W_0^{1,p'}(G;\hc)\subset L_{p'}(G;\hc) \rightarrow L_{p'}(G;\hc)$.
\end{proposition}
\begin{proof}
	By the definition of weak $\mathbf{D}$-derivative, it is clear that $W_0^{1,p'}(G;\hc) \subset \operatorname{dom}(\mathbf{D}^T)$ and $\mathbf{D}^Tv= \mathbf{D}v$ for $v\in W_0^{1,p'}(G;\hc)$. On the other hand, suppose $v\in \operatorname{dom}(\mathbf{D}^T)$ and $u\in \mathfrak{D}_p(G;\hc)$. Using the relations $\sc(pe_k)=-p_k$, $k=1,2,3$, and the fact that $\mathfrak{D}_p(G;\hc)$ is a right $\hc$-module, we have
	\begin{align*}
		\int_G\overline{v}\mathbf{D}u & = (v|\mathbf{D}u)_G-\sum_{k=1}^{3}(v|(\mathbf{D}u)e_k)_Ge_k= (\mathbf{D}^Tv|u)_G-\sum_{k=1}^{3}(v|\mathbf{D}(ue_k))_Ge_k\\
		&= (\mathbf{D}^Tv|u)_G-\sum_{k=1}^{3}(\mathbf{D}^Tv|ue_k)_Ge_k= \int_G \overline{\mathbf{D}^Tv}u.
	\end{align*}
In particular, the equality is valid for all $u\in C_0^{\infty}(G;\hc)$, and hence $v\in \mathfrak{D}_{p'}(G;\hc)$ with $\mathbf{D}^Tv=\mathbf{D}v$. Consequently, $v=\mathbf{T}_G\mathbf{D}^Tv+g$, with $g\in \mathfrak{M}^{p'}(G;\hc)$. By Remarks \ref{remark:weakderi}(iii) and  \ref{remark:transposesumandproduct}, and  Proposition \ref{prop:adjointsofsomeop}(i), $\mathbf{T}_G\mathbf{D}^T\subset (\mathbf{D}\mathbf{T}_G)^T=\mathbf{I}$. Thus, $v=v+g$ and $g\equiv 0$. Hence $v=\mathbf{T}_G\mathbf{D}^Tv$ and by  Proposition \ref{prop:integralop}(i), $v\in W^{1,p'}(G;\hc)$. Take $\psi\in W^{1-\frac{1}{p},p}(\Gamma;\hc)$ and $\Psi\in W^{1,p}(G;\hc)$ with $\operatorname{tr}_{\Gamma}\Psi=\psi$. Using the Gauss formula and the previous computation, we get:
\begin{align*}
	\int_G\operatorname{tr}_{\Gamma}\overline{v} \widehat{\nu} \psi dS= \int_G \left(-\overline{\mathbf{D}v}\Psi+\overline{v}\mathbf{D}\Psi\right)=0
\end{align*}
Taking the conjugate, we obtain $\int_G\overline{\psi}\widehat{\nu}\operatorname{tr}_{\Gamma}v dS=0$ for all $\psi\in W^{1-\frac{1}{p},p}(\Gamma;\hc)$. According to Lemma \ref{lema:integral}, $v=v_0+g$, where $v_0\in W_0^{1,p'}(G;\hc)$ and $g\in \mathfrak{M}^{p'}(G;\hc)\cap W^{1,p'}(G;\hc)$. Given $u\in W^{1,p}(G;\hc)$, we obtain
\begin{align*}
	\int_G \overline{(v_0+g)}\mathbf{D}u=\int_G \overline{(\mathbf{D}v_0+\mathbf{D}g)}u = \int_{G}\overline{\mathbf{D}v_0}u=\int_G \overline{v_0}\mathbf{D}u,
\end{align*}
from which we concluded that $\int_{G}\overline{g}\mathbf{D}u=0$ for all $u\in W^{1,p}(G;\hc)$. Taking $\varphi\in C_0^{\infty}(G;\hc)$ arbitrarily and using $u=\mathbf{T}_G\varphi\in W^{1,p}(G;\hc)$ along with Proposition \ref{prop:integralop}(i), we obtain that $\int_G \overline{g}\varphi=0$. Since $\varphi\in C_0^{\infty}(G;\hc)$ was chosen arbitrarily, we conclude that $g=0$ and $v=v_0\in W^{1,p'}_0(G;\hc)$. Therefore $\operatorname{dom}(\mathbf{D}^T)=W_0^{1,p'}(G;\hc)$, as desired.
\end{proof}
\begin{proposition}\label{prop:transposeoperatorDQ}
	For $1<p<\infty$, the transpose of $\mathbf{D}+\mathbf{Q}_A: \mathfrak{D}_p(G;\hc)\subset L_p(G;\hc)\rightarrow L_p(G;\hc)$ is given by $\mathbf{D}-\mathbf{Q}_{A^{\star}}: W_0^{1,p'}(G;\hc) \subset L_{p'}(G;\hc) \rightarrow L_{p'}(G;\hc)$.
	
	For $p=2$, the (Hilbert) adjoint is given by $\mathbf{D}-\mathbf{Q}_{A^{\dagger}}: W_0^{1,2}(G;\hc)\subset L_2(G;\hc)\rightarrow L_2(G;\hc)$, where $A^{\dagger}=(a_1^{\dagger},a_4^{*},a_3^{\dagger}.a_2^{*})$.
\end{proposition}
\begin{proof}
	It follows from Proposition \ref{prop:adjointsofsomeop}(ii), together with the fact that $\mathbf{Q}_A\in \mathcal{B}(L_p(G;\hc))$, and Remark \ref{remark:transposesumandproduct}. The proof that $\mathbf{Q}_A^*=\mathbf{Q}_{A^{\dagger}}$ is similar to that of Proposition \ref{prop:adjointsofsomeop}(ii).
\end{proof}

Let $X$ be a Banach space. Given $M\subset X$ and $N\subset X^{\star}$, the annihilator and the pre-annihilator of $M$ and $N$ are denoted by
\[
M^a:= \{\varphi\in X^{\star} | \forall x\in M\; (\varphi|x)_X=0 \}\mbox{ and } N_a:= \{x\in X | \forall \varphi\in N\; (\varphi|x)_X=0 \},
\] 
respectively. The following result is presented in \cite[Remark 17 from Ch.2]{brezis} without proof. We include it for the sake of completeness.

\begin{lemma}\label{remark:annihilator}

	Let $\mathbf{A}: \operatorname{dom}(A)\subset X\rightarrow X$ be a densely defined operator. If $X$ is reflexive, then $(\ker \mathbf{A})^a$ is the closure of the image of $\mathbf{A}^T$ in the strong topology.
\end{lemma}
\begin{proof}
Let $R$ be the image of $\mathbf{A}^T$.	By \cite[Cor. 2.18]{brezis}, $\ker \mathbf{A} = R_a$. Thus, $(\ker \mathbf{A})^a=(R_a)^a$. Let us show that $(R_a)^a=(R^a)_a$. Since $X$ is reflexive, the canonical embedding $J:X\rightarrow X^{\star \star}$, given by $(Jx|\varphi)_{X^{\star}}=(\varphi|x)_X=0$ for all $\varphi\in X^{\star}$, is an isometric isomorphism. Note that $x\in R_a$ iff $Jx\in R^a$. Indeed, this follows from the relation $(Jx|\varphi)_{X^{\star}}=(\varphi|x)_X$ for all $\varphi\in R$. Since $J$ is an isomorphism, $R^a=J(R_a)$.
Thus, 
\[
 \varphi\in (R_a)^a \Leftrightarrow (\varphi|x)_X=0 \mbox{ for all } x \in R_a \Leftrightarrow (\Phi|\varphi)_{X^{\star}}=0 \mbox{ for all } \Phi\in R^a \Leftrightarrow \varphi \in (R^a)_a.
\]
By \cite[Prop. 1.9]{brezis}, $(R^a)_a=\overline{R}^{X^{\star}}$ (the closure in the strong topology), and we conclude that $(\ker \mathbf{A})^a=\overline{R}^{X^{\star}}$, as desired.
\end{proof}

As a consequence, we obtain the following result.
\begin{theorem}\label{th:anihilator}
	For $1<p<\infty$, we have
	\begin{equation}\label{eq:annhilatorformula}
	(\mathfrak{V}_A^p(G;\hc))^a= \overline{(\mathbf{D}-\mathbf{Q}_{A^{\star}})W_0^{1,p'}(G;\hc)}^{L_{p'}}.	
	\end{equation}
	  For $p=2$, the following orthogonal decomposition holds 
	\begin{equation}\label{eq:orthogonaldecomp}
		L_2(G;\hc)= \mathfrak{V}_A^2(G;\hc)\oplus \overline{(\mathbf{D}-\mathbf{Q}_{A^{\dagger}})W_0^{1,2}(G;\hc)}^{L_2}.
	\end{equation}
\end{theorem}
\begin{proof}
	The first part follows from Proposition \ref{prop:transposeoperatorDQ} and Lemma \ref{remark:annihilator}. The second part is a consequence of Proposition \ref{prop:transposeoperatorDQ} and \cite[Prop. 16(ii) and Th. 1.8 (iii)]{schmudgen}.
\end{proof}

\begin{example}\label{eq:examplemainvekua}
	Let $f\in W^{2,\infty}(G;\hc)$ be a scalar function such that $\frac{1}{f}\in L_{\infty}(G)$, and consider $A=\left(-\frac{\nabla f}{f},0,0,0\right)$. Denote $\mathbf{D}_f:= \mathbf{D}-\mathbf{Q}_A=\mathbf{D}+M^{\frac{\nabla f}{f}}$. This is one of the so-called {\it main Vekua operators} \cite[Ch. 16]{pseudoanalyticvlad}. Hence 
	\[
	\mathbf{D}_f^T=\mathbf{D}+M^{\overline{\frac{\nabla f}{f}}}= \mathbf{D}+M^{-\frac{\nabla f}{f}}=\mathbf{D}+M^{f\nabla\left(\frac{1}{f}\right)}=\mathbf{D}_{\frac{1}{f}}.
	\]
	 Note that $\mathbf{D}_{\frac{1}{f}}=\mathbf{D}-M^{\frac{\nabla f}{f}}$. Thus, $(\mathfrak{V}^p_{\left(-\frac{\nabla f}{f},0,0,0\right)}(G;\hc))^a=\overline{(\mathbf{D}-M^{\frac{\nabla f}{f}})W_0^{1,p'}(G;\hc)}^{L_{p'}}$. Suppose that $w\in W^{1,p'}(G;\hc)$  satisfies  $\mathbf{D}_{\frac{1}{f}}w\in \mathfrak{D}_{p'}(G;\hc)$. Thus,
	\begin{align*}
		\mathbf{D}_f\mathbf{D}_{\frac{1}{f}} w= & \mathbf{D}\left(\mathbf{D}w-w\frac{\nabla f}{f}\right)+\left(\mathbf{D}w-w\frac{\nabla f}{f}\right)\frac{\nabla f}{f}\\
		= & -\bigtriangleup w-\mathbf{D}\left(w\frac{\nabla f}{f}\right)+(\mathbf{D}w)\frac{\nabla f}{f}-w\left(\frac{\nabla f}{f}\right)^2\\
		= & -\bigtriangleup w -(\mathbf{D}w)\frac{\nabla f}{f}-\overline{w}\mathbf{D}\left(\frac{\nabla f}{f}\right)+2\sum_{j=1}^{3}w_j\frac{\partial}{\partial x_j}\left(\frac{\nabla f}{f}\right)+(\mathbf{D}w)\frac{\nabla f}{f}-w\left(\frac{\nabla f}{f}\right)^2\\
		= & -\bigtriangleup w-\overline{w}\left(-\left(\frac{\nabla f}{f}\right)^2-\frac{\bigtriangleup f}{f}\right)+2\sum_{j=1}^{3}w_j\frac{\partial}{\partial x_j}\left(\frac{\nabla f}{f}\right)-w\left(\frac{\nabla f}{f}\right)^2\\
		=& -\bigtriangleup w+\frac{\bigtriangleup f}{f} \overline{w}-2\left(\frac{\nabla f}{f}\right)^2\vec{w}+2\sum_{j=1}^{3}w_j\frac{\partial}{\partial x_j}\left(\frac{\nabla f}{f}\right).
	\end{align*}
	Let $q_f=\frac{\bigtriangleup f}{f}$ and note that $q_{\frac{1}{f}}=f\bigtriangleup\left(\frac{1}{f}\right)= 2\frac{\nabla f\cdot \nabla f}{f^2}-\frac{\bigtriangleup f}{f}=-2\left(\frac{\nabla f}{f}\right)^2-\frac{\bigtriangleup f}{f}$. The potential $q_{\frac{1}{f}}$ is the Darboux transform of $q_f$. Let $R_{j,k}=\frac{1}{f}\frac{\partial f}{\partial x_k\partial x_j}-\frac{1}{f^2}\frac{\partial f}{\partial x_k}\frac{\partial f}{\partial x_j}$ for $j,k=1,2,3$. Observe that the matrix $R=(R_{j,k})_{j,k=1}^{3}$ is symmetric and $\sum_{j=1}^{3}w_j\frac{\partial}{\partial x_j}\left(\frac{\nabla f}{f}\right)=\sum_{k,j=1}^{3}R_{k,j}w_je_k=R\vec{w}$. Consequently, 
	\begin{align}
		\operatorname{Sc}\mathbf{D}_f\mathbf{D}_{\frac{1}{f}}w= & (-\bigtriangleup+q_f)\operatorname{Sc}w,\label{eq:schrodinger1}\\
		\operatorname{Vec}\mathbf{D}_f\mathbf{D}_{\frac{1}{f}}w=& (-\bigtriangleup+q_{\frac{1}{f}})\vec{w}+2R\vec{w}.\label{eq:schrodinger2}
	\end{align}
	Equation \eqref{eq:schrodinger1} is known as a factorization of the Schr\"odinger operator, see \cite{kravchenkoricatti} or \cite[Th. 154]{pseudoanalyticvlad}. Consequently, $w=(\mathbf{D}-M^{\frac{\nabla f}{f}})u$ with $u\in W_0^{1,p'}(G;\hc)$ belongs to $\mathfrak{V}^{p'}_{\left(-\frac{\nabla f}{f},0,0,0\right)}(G;\hc)$ if $u$ is a solution of system \eqref{eq:schrodinger1}-\eqref{eq:schrodinger2}.
\end{example}
\begin{remark}\label{remark:invertibilityS}
	For fixed $x,y\in G$, a direct computation shows the inequality 
	\[
	 \max\left\{\int_G|E(y-x)|dV_y,  \int_G|E(y-x)|dV_x\right\}\leqslant \operatorname{diam}(G).
	\]
	By \cite[Th. 6.18]{folland} and \eqref{normproduct}, we have  $\|\mathbf{T}_Gu\|_{L_p(G;\hc)}\leqslant \sqrt{2}\operatorname{diam}(G)\|u\|_{L_p(G;\hc)}$ for all $u\in L_p(G;\hc)$. Thus, $\|\mathbf{T}_G\|_{\mathcal{B}(L_p(G;\hc))}\leqslant \sqrt{2}\operatorname{diam}(G)$. On the other hand, by \eqref{normproduct}, we have $\|\mathbf{Q}w\|_{L_p(G;\hc)}\leqslant \sqrt{2}\left(\sum_{j=1}^4\|a_j\|_{L_{\infty}(G;\hc)}\right)\|w\|_{L_p(G;\hc)}$. Consequently, if $\sum_{j=1}^{4}\|a_j\|_{L_{\infty}(G;\hc)}<\frac{1}{2\operatorname{diam}(G)}$, then $\|\mathbf{T}_G\mathbf{Q}_A\|_{\mathcal{B}(L_p(G;\hc)}<1$, which implies that $\mathbf{S}_G^A\in \mathcal{G}(L_p(G;\hc))$ \cite[Lemma 2.1 of Ch. VII]{conwayfunctional}.
\end{remark}

\begin{theorem}\label{th:adjointforboundedcoefficients}
	If $\sum_{j=1}^{4}\|a_j\|_{L_{\infty}(G;\hc)}<\frac{1}{2\operatorname{diam}(G)}$, then 
	\begin{equation}\label{eq:annihilatorformula2}
	(\mathfrak{V}_A^p(G;\hc))^a= (\mathbf{D}-\mathbf{Q}_{A^{\star}})W_0^{1,p'}(G;\hc).	
	\end{equation}
	 For $p=2$ we have the orthogonal decomposition 
	\begin{equation}\label{eq:orthogonaldecomp2}
		L_2(G;\hc)= \mathfrak{V}_A^2(G;\hc)\oplus (\mathbf{D}-\mathbf{Q}_{A^{\dagger}})W_0^{1,2}(G;\hc).
	\end{equation}
\end{theorem}
\begin{proof}
By Remark \ref{remark:invertibilityS}, $\mathbf{S}_G^{A^{\star}}\in \mathcal{G}(L_{p'}(G;\hc))$.  Consider $\mathbf{R}=(\mathbf{S}_G^{A^{\star}})^{-1}\mathbf{T}_G\in \mathcal{B}(L_{p'}(G;\hc))$.  For $u\in W^{1,p'}_0(G;\hc)$, Proposition \eqref{prop:integralop}(iii) implies
\[
\mathbf{R}(\mathbf{D}-\mathbf{Q}_{A^{\star}})u=(\mathbf{S}_G^{A^{\star}})^{-1}(\mathbf{T}_G\mathbf{D}u-\mathbf{T}_G\mathbf{Q}_{A^{\star}}u)=(\mathbf{S}_G^{A^{\star}})^{-1}\mathbf{S}_G^{A^{\star}}u=u.
\]
 Consequently, $\|u\|_{L_{p'}(G;\hc)}\leqslant \|\mathbf{R}\|_{\mathcal{B}(L_{p'}(G;\hc))}\|(\mathbf{D}-\mathbf{Q}_{A^{\star}})u\|_{L_p(G;\hc)}$. Therefore,\\ $\|(\mathbf{D}-\mathbf{Q}_{A^{\star}})u\|_{L_p(G;\hc)}\geqslant c\|u\|_{L_{p'}(G;\hc)}$ for all $u\in W_0^{1,2}(G;\hc)$, with $c=\|\mathbf{R}\|_{\mathcal{B}(L_{p'}(G;\hc))}^{-1}$. Hence, Propositions \ref{prop:operadorDQisclosed}, \ref{prop:transposeoperatorDQ}, and \cite[Th. 2.20]{brezis} imply that $(\mathbf{D}-\mathbf{Q}_{A^{\star}})W_0^{1,p'}(G;\hc)$ is closed in $L_{p'}(G;\hc)$.
\end{proof}

For the particular case $A=(0,0,0,0)$, we have obtain the annihilator of the monogenic Bergman space and the standard Hodge decomposition \cite[Th. 8.7]{gurlebeckclifford}.
\begin{corollary}
	For $1<p<\infty$, $(\mathfrak{M}^p(G;\hc))^a=\mathbf{D}W_0^{1,p'}(G;\hc)$. For the case $p=2$, we obtain the Hodge decomposition $L_2(G;\hc)=\mathfrak{M}^2(G;\hc)\oplus \mathbf{D}W_0^{1,2}(G;\hc)$. The decomposition is also valid with respect to the $\hc$-inner product \eqref{eq:hcinnerproduct}.
\end{corollary}
For some classes of coefficients $A$, we can obtain the result of Theorem \ref{th:adjointforboundedcoefficients} without the restriction in $A$.
\begin{example}\label{example:helmholtz}
	In the case $A=(\pm\alpha,0,0,0)$ where $\alpha\in \mathbb{C}$ is a scalar with $\operatorname{Im}\alpha\geqslant 0$, we have  $\mathbf{D}-\mathbf{Q}_A=\mathbf{D}\mp\alpha \mathbf{I}=:\mathbf{D}_{\mp\alpha}$. It is known that the function 
	\[
	\mathcal{K}_{\mp\alpha}(x)=\frac{-e^{\operatorname{i}\alpha |\vec{x}|}}{4\pi}\left(\mp\frac{\alpha}{|\vec{x}|}+\frac{\vec{x}}{|\vec{x}|^3}-\operatorname{i}\alpha\frac{\vec{x}}{|\vec{x}|^2}\right), \quad x\neq 0.
	\]
	is a fundamental solution for the operator $\mathbf{D}_{\mp\alpha}$ \cite[Sec. 2.5]{kravquaternionic}. The operator 
	\[
	\mathbf{T}_G^{\mp\alpha}w(x)=\int_G \mathcal{K}_{\mp\alpha}(x-y)w(y)dV_y,
	\]
	possesses properties similar to the Theodorescu operator. In particular, for $w\in C_0^{\infty}(G;\hc)$, the following relation holds \cite[Sec. 2.5, Th. 4]{kravquaternionic}
	\begin{equation}\label{eq:BPHelmholtz}
		w(x)=\mathbf{T}_G^{\mp\alpha} \mathbf{D}_{\mp\alpha} w(x),\quad x\in G.
	\end{equation}
A direct computation yields the bound
\[
\max\left\{\int_G|\mathcal{K}_{\mp\alpha}(x-y)|dV_x, \int_G |\mathcal{K}_{\mp\alpha}(x-y)|dV_y\right\}\leqslant C,.
\]
where $C=2|\alpha|\left(\frac{1-e^{-d\operatorname{Im}\alpha}}{\operatorname{Im}^2\alpha}-d\frac{e^{d\operatorname{Im}\alpha}}{\operatorname{Im}\alpha}\right)+\frac{1-e^{d\operatorname{Im}\alpha}}{\operatorname{Im}\alpha}$ and $d=\operatorname{diam}(G)$. This expression is well defined even for $\operatorname{Im}\alpha=0$. Proceeding as in Remark \ref{remark:invertibilityS}, we conclude that $\mathbf{T}_G^{\mp\alpha}\in \mathcal{B}(L_p(G;\hc))$ for $1<p<\infty$. In this case, $\mathbf{D}_{\mp\alpha}^T=\mathbf{D}_{\mp\alpha}: W_0^{1,p'}(G;\hc)\rightarrow L_{p'}(G;\hc)$. From \ref{eq:BPHelmholtz}, we derive
\[
\|\mathbf{D}_{\mp\alpha} w\|_{L_{p'}(G;\hc)}\geqslant \|\mathbf{T}_G^{\mp\alpha}\|_{\mathcal{B}(L_{p'}(G;\hc))}\|w\|_{L_{p'}(G;\hc)}, \quad \forall w\in C_0^{\infty}(G;\hc).
\]
Since $\mathbf{T}_G^{\mp\alpha}$ is continuous and $C_0^{\infty}(G;\hc)$ is dense in $W_0^{1,p'}(G;\hc)$, this inequality holds for all $w\in W_0^{1,p'}(G;\hc)$. By \cite[Th. 2.20]{brezis}, we conclude that $\mathbf{D}_{\mp\alpha}W_0^{1,p'}(G;\hc)$ is closed in $L_{p'}(G;\hc)$, implying 
\begin{equation}\label{eq:anihilatorHelmholtz}
	(\mathfrak{V}_{(\pm\alpha,0,0,0)}^p(G;\hc))^a=\mathbf{D}_{\mp\alpha}W_0^{1,p'}(G;\hc).
\end{equation}
The space $\mathfrak{V}_{(\mp\alpha,0,0,0)}^p(G;\hc)$ is a right $\hc$-module, and for $p=2$, the orthogonal decomposition
\[
L_2(G;\hc)=\mathfrak{V}_{(\mp\alpha,0,0,0)}^2(G;\hc)\oplus \mathbf{D}_{\pm(-\alpha^*)}W_0^{1,2}(G;\hc),
\]
is valid with respect to the $\hc$-inner product \eqref{eq:hcinnerproduct} (note that $\operatorname{Im}(-\alpha^*)\geqslant 0$).
\end{example}

\begin{example}
	For a general $\alpha\in \hc$, set $\lambda =\sqrt{\vec{\alpha}^2}$ to be chosen such that $\operatorname{Im}\lambda \geqslant 0$. Define $\xi_{\pm}=\alpha_0\pm \lambda$ and $P^{\pm}=M^{\lambda\pm \vec{\alpha}}$. Denote by $\mathfrak{S}$ the set of zero divisors of $\hc$. In this case, $\mathbf{D}-\mathbf{Q}_{(-\alpha,0,0,0)}=\mathbf{D}+M^{\alpha}=:\mathbf{D}_{\alpha}$ and $D_{\alpha}^T= D+M^{\overline{\alpha}}$. Consider the integral operator
	\begin{equation}\label{eq:Thodorescualphabiqua}
		\mathbf{T}_G^{\alpha}:=\begin{cases}
			P^+\mathbf{T}_G^{\xi_{+}}+P^{-}\mathbf{T}_G^{\xi_{-}}, &\mbox{ if } \alpha\not\in \mathfrak{S} \mbox{ and } \vec{\alpha}^2\neq 0,\\
			\left(\mathbf{I}+M^{\vec{\alpha}}\frac{\partial}{\partial \alpha_0}\right)\mathbf{T}_G^{\alpha_0}, & \mbox{ if } \alpha \not \in \mathfrak{S} \mbox{ and } \vec{\alpha}^2=  0,\\
			P^{+}\mathbf{T}_G^{2\alpha_0}+P^{-}\mathbf{T}_G^0, & \mbox{ if } \alpha\in \mathfrak{S} \mbox{ and } \alpha_0\neq 0,\\
			\mathbf{T}_G+M^{\alpha}L_G, & \mbox{ if } \alpha\in \mathfrak{S} \mbox{ and } \alpha_0=0,
		\end{cases}
	\end{equation}
where $L_Gw(x)=\frac{1}{4\pi}\int_G \frac{w(y)}{|x-y|}dV_y$ is the Newtonian potential. Applying the same procedure as in Remark \ref{remark:invertibilityS} and the previous example, we find that $\mathbf{T}_G^{\alpha}\in \mathcal{B}(L_p(G;\hc))$ for $1<p<\infty$. According to \cite[Th. 4.17]{shapirokrav} , we have
\[
w(x)=\mathbf{T}^{\alpha}_G\mathbf{D}_{\alpha}w(x),\quad x\in G,
\]
for all $w\in C_0^{\infty}(G;\hc)$. Using the same reasoning as in the previous example, we obtain
\begin{equation}\label{eq:anihilatorHelmholtzbiqua}
	(\mathfrak{V}_{(-\alpha,0,0,0)}^p(G;\hc))^a=\mathbf{D}_{\overline{\alpha}}W_0^{1,p'}(G;\hc).
\end{equation} 
\end{example}
\begin{example}
	Let $f\in W^{2,\infty}(G)$ as in Example \ref{eq:examplemainvekua}, and take  $A=\left(0,0,\frac{\nabla f}{f},0\right)$. Hence, by \eqref{eq:leibnizrule}
	\[
	\mathbf{D}-\mathbf{Q}_A= \mathbf{D}-\frac{\nabla f}{f}\mathbf{I}= f\mathbf{D}\frac{1}{f}.
	\]
	In this case,
	\[
	(\mathbf{D}-\mathbf{Q}_A)^T= \mathbf{D}+\frac{\nabla f}{f}\mathbf{I}= \frac{1}{f}\mathbf{D}f.
	\]
	Note that the operator $L_p(G;\hc) \ni w\mapsto fw\in L_p(G;\hc)$ is a homeomorphism with inverse $L_p(G;\hc) \ni w\mapsto \frac{w}{f}\in L_p(G;\hc)$. Actually, $w\in \mathfrak{V}^p_{\left(0,0,\frac{\nabla f}{f},0\right)}(G;\hc)$ iff $\frac{w}{f}\in \mathfrak{M}^p(G;\hc)$. Hence, the operator $f\mathbf{T}_G\frac{1}{f}$ satisfies the same properties as in  Proposition \ref{prop:integralop} for the operator $f\mathbf{D}\frac{1}{f}$. Using the same reasoning as in Example \ref{example:helmholtz}, we conclude that
	\begin{equation}\label{eq:annihilatorDf}
		\left(\mathfrak{V}^p_{\left(0,0,\frac{\nabla f}{f},0\right)}(G;\hc)\right)^a= \left(\frac{1}{f}\mathbf{D}f\right)W_0^{1,p'}(G;\hc).
	\end{equation}
	For $f$ real-valued and $p=2$, we obtain 
	\begin{equation}\label{eq:hodgemainVekua}
		L_2(G;\hc)= \mathfrak{V}^2_{\left(0,0,\frac{\nabla f}{f},0\right)}(G;\hc)\oplus\left(\frac{1}{f}\mathbf{D}f\right)W_0^{1,2}(G;\hc).
	\end{equation}
	It is important to note that $\mathfrak{V}^p_{\left(0,0,\frac{\nabla f}{f},0\right)}(G;\hc)$ is a right $\hc$-module, and the decomposition \eqref{eq:hodgemainVekua} is valid with the $\hc$-inner product  \eqref{eq:hcinnerproduct}.

	Now, suppose that $u\in W^{1,p'}(G;\hc)$ satisfies $\frac{1}{f}\mathbf{D}(fu)\in \mathfrak{D}_{p'}(G;\hc)$. We have
	\begin{align*}
		\left(f\mathbf{D}\frac{1}{f}\right)	\left(f\mathbf{D}\frac{1}{f}\right)^Tu =& \left(f\mathbf{D}\frac{1}{f}\right)\left(\frac{1}{f}\mathbf{D}fu\right)\\
		=& \frac{1}{f}\mathbf{D}\left[\frac{\nabla f}{f^2}u+\frac{\mathbf{D}u}{f}\right]  \\
		=& f\Big{[}-\frac{\bigtriangleup u}{f}-\frac{\nabla f \mathbf{D}u}{f^2}+\left(-2\frac{(\nabla f)^2}{f^3}-\frac{\bigtriangleup f}{f^2}\right)u\\
		& \quad -\frac{\nabla f\mathbf{D}u}{f^2}-\frac{2}{f^2}\sum_{k=1}^{3}\frac{\partial f}{\partial x_k}\frac{\partial u}{\partial x_k}\Big{]}. 
	\end{align*}
	
	Note that 
	\[
	(\nabla f)( \mathbf{D}u)= -\sum_{k=1}^{3}\frac{\partial f}{\partial x_k}\frac{\partial u}{\partial x_k}+\sum_{\overset{k,j=1}{k\neq j}}^{3}e_ke_j\frac{\partial f}{\partial x_k}\frac{\partial u}{\partial x_k}.
	\]
	Consequently
	\begin{equation}\label{eq:factorizationDf}
		\left(f\mathbf{D}\frac{1}{f}\right)	\left(f\mathbf{D}\frac{1}{f}\right)^Tu =(-\bigtriangleup+q_{\frac{1}{f}})u-\frac{2}{f}\sum_{\overset{k,j=1}{k\neq j}}^{3}e_ke_j\frac{\partial f}{\partial x_k}\frac{\partial u}{\partial x_k}.
	\end{equation}
	For scalar-valued $u$, this yields $\left(f\mathbf{D}\frac{1}{f}\right)	\left(f\mathbf{D}\frac{1}{f}\right)^Tu =(-\bigtriangleup+q_{\frac{1}{f}})u-\frac{2\nabla f\times \nabla u}{f}$.

	A similar treatment can be applied to the operator $\mathbf{D}-\frac{\nabla f}{f}\mathbf{I}-\mu \mathbf{I}= f(\mathbf{D}-\mu \mathbf{I})\frac{1}{f}$, with $\mu \in \mathbb{C}$ satisfying $\operatorname{Im}\mu \geqslant 0$ (see \cite{sprossig2}).
\end{example}

	Theorem \ref{th:anihilator} and decomposition \ref{eq:orthogonaldecomp} motivated us to consider a decomposition of the form  $L_{p'}(G;\hc)=\mathfrak{V}_A^{p'}(G;\hc)\oplus \overline{(\mathbf{D}-\mathbf{Q}_{A^{\star}})W_0^{1,p'}(G;\hc)}^{L_{p'}}$. If such a decomposition holds, then $(\mathfrak{V}_A^p(G;\hc))^{\star}=\mathfrak{V}_A^{p'}(G;\hc)$. However, establishing a result of this form is not  straightforward, even for Bergman spaces of analytic functions. For $p>2$, following and a idea from \cite{cofman,hedenmal}, we can establish the relationship:
	\begin{equation}\label{eq:Lpdecomposition}
		L_{p'}(G;\hc)=\overline{\mathfrak{V}_A^{p'}(G;\hc)+\overline{(\mathbf{D}-\mathbf{Q}_{A^{\star}})W_0^{1,p'}(G;\hc)}^{L_{p'}}}^{L_{p'}}.
	\end{equation}
Since $p'<2$, the embeddings $L_2(G;\hc)\hookrightarrow L_{p'}(G;\hc)$ and $L_p(G;\hc)\hookrightarrow L_2(G;\hc)$ are bounded. Consequently, $V_A^2(G;\hc)\subset V_A^{p'}(G;\hc)$ and $(V_A^2(G;\hc))^a\subset (V_A^p(G;\hc))^a$. Note that $L_2(G;\hc)$ (as a subspace) is dense in $L_{p'}(G;\hc)$. Thus,
\[
L_2(G;\hc)= V_A^2(G;\hc)\oplus (V_A^2(G;\hc))^a\subset V_A^{p'}(G;\hc)+(V_A^p(G;\hc))^a.
\]
By the density of $L_2(G;\hc)$ and Theorem \ref{th:anihilator}, we obtain \eqref{eq:Lpdecomposition}. A priori, the sum in \eqref{eq:Lpdecomposition} does not have to be direct. 

\begin{example}\label{example:monogenicdecomposition}
	In the case $A=(0,0,0,0)$, \eqref{eq:Lpdecomposition} takes the form 
	\[
	L_{p'}(G;\hc) = \overline{\mathfrak{M}^{p'}(G;\hc)+\mathbf{D}W_0^{1,p'}(G;\hc)}^{L_{p'}}.
	\]
	Thus, $w=\mathbf{D}u$, with $u\in W_0^{1,p'}(G;\hc)$, belongs to $\mathfrak{M}^{p'}(G;\hc)$ iff $\bigtriangleup u=0$. If $\Gamma$ is of class $C^{1,1}$, then $u=0$ \cite[Th. 9.15]{gilbargtrudinger}, and the sum is direct.
\end{example}

\begin{example}
Now consider $G=\mathbb{B}$, the unit ball, and $A=(\operatorname{i}\sqrt{\lambda} e_1,0,0,0)$  with $\lambda>0$. Take $\mathbf{D}-\mathbf{Q}_A=\mathbf{D}-M^{\vec{\alpha}}$, with $\vec{\alpha}=\operatorname{i}\sqrt{\lambda}e_1$. Hence, $(\mathbf{D}-M^{\vec{\alpha}})^T=\mathbf{D}+M^{\vec{\alpha}}$. When $w=(\mathbf{D}+M^{\vec{\alpha}})u$, $u\in W_0^{1,p'}(\mathbb{B};\hc)$, it belongs to $\mathfrak{V}_A^{p'}(\mathbb{B};\hc)$  if
\[
0=(\mathbf{D}-M^{\vec{\alpha}})(\mathbf{D}+M^{\vec{\alpha}})u= \mathbf{D}^2u+(\mathbf{D}u)\vec{\alpha}-(\mathbf{D}u)\vec{\alpha}-u\vec{\alpha}\vec{\alpha}=-\bigtriangleup u-\lambda u.
\]
Let $\sqrt{\lambda}$ be equal to the first positive root of the Bessel function $J_{\frac{1}{2}}(z)$, and let   $u_{\lambda}(x)=\frac{J_{\frac{1}{2}}\left(\sqrt{\lambda}|x|\right)}{\sqrt{|x|}}$. According to \cite[Ch. 10]{abramowitz}, this function is entire in $r=|x|$, hence $u_{\lambda}\in W_0^{1,p'}(\mathbb{B})\cap W^{2,p'}(\mathbb{B})$,  and by \cite[pp. 344-345]{vladimirov}, $-\bigtriangleup u_{\lambda}=\lambda u_{\lambda}$. Consequently,  $w=(\mathbf{D}+M^{\vec{\alpha}})u_{\lambda}=\nabla u_{\lambda}+\operatorname{i}\sqrt{\lambda} u_{\lambda}e_1$ is a nontrivial function that belongs to $\mathfrak{V}_A^{p'}(\mathbb{B};\hc)\cap (\mathbf{D}+M^{\vec{\alpha}})W_0^{1,p'}(\mathbb{B};\hc)$.
\end{example}

\section{Conclusions}

The main properties of the Bergman space of weak $L_p$-solutions of a class of quaternionic Vekua equations have been established. The regularity of the solutions was analyzed, demonstrating their b-locally H\"older continuity. Furthermore, it was shown that the evaluation map is bounded with respect to the $L_p$-norm at each point in the domain. For the case $p=2$, the existence of a reproducing kernel and an explicit formula for the orthogonal projection were demonstrated. Additionally, the explicit form of the annihilator of the Bergman space for $1<p<\infty$ was computed, yielding an orthogonal decomposition for $L_2(G;\hc)$.

For $p\neq 2$, an interesting question is to compute the dual $(V_A^p(G;\hc))^{\star}$. A conjecture could be that $(V_A^p(G;\hc))^{\star}=V_A^{p'}(G;\hc)$, which is equivalent to showing that $L_{p'}(G;\hc)= V_A^{p'}(G;\hc)\oplus (V_A^p(G;\hc))^a$, or that the Bergman projection $\mathbf{P}_G^A$ can be extended continuously to $L_{p'}(G;\hc)$. However, this is not a trivial question, even for the case of monogenic functions $\mathfrak{M}^p(G;\hc)$ (see \cite{cofman,hedenmal} for the case of the classical analytic and harmonic Bergman space).

\section*{Acknowledgments}
The author thanks to  Instituto de Matem\'aticas de la U.N.A.M. (M\'exico), where this work was developed, and CONAHCYT for their support through the program {\it Estancias Posdoctorales por México Convocatoria 2023 (I)},

		
\end{document}